\documentclass[11pt]{article}
\usepackage{booktabs}
\usepackage{graphicx}
\usepackage{caption}
\usepackage{subcaption}
\usepackage{enumitem}
\usepackage[a4paper,margin=1in]{geometry}
\usepackage{amsmath,amssymb,amsthm,mathrsfs,bm}
\usepackage{float}
\usepackage{booktabs}
\usepackage{tikz}
\usetikzlibrary{arrows.meta,calc,patterns,decorations.markings}
\usepackage{cite}
\newtheorem{theorem}{Theorem}[section]
\newtheorem{lemma}[theorem]{Lemma}
\newtheorem{proposition}[theorem]{Proposition}
\newtheorem{remark}{Remark}[section]

\numberwithin{equation}{section}

\newcommand{\Rp}{\mathbb{R}_+}
\newcommand{\Mplus}{M_+}
\newcommand{\mB}{\mathrm{B}}
\newcommand{\me}{\mathrm{e}}

\definecolor{figred}{RGB}{215,35,35}
\definecolor{figgreen}{RGB}{35,150,90}
\definecolor{figblue}{RGB}{75,120,190}
\usepackage[colorlinks=true,linkcolor=blue,citecolor=blue,urlcolor=blue]{hyperref}
\tikzset{
  axis/.style={->,thick},
  guide/.style={gray,dashed,thin},
  tangent/.style={gray!70,dashed,semithick},
  physical/.style={gray,densely dashed,thin},
  phase/.style={very thick,figred,postaction={decorate},
    decoration={markings,mark=at position 0.62 with {\arrow{Stealth[length=2mm]}}}},
  phaseb/.style={very thick,figred,postaction={decorate},
    decoration={markings,mark=at position 0.78 with {\arrow{Stealth[length=2mm]}}}},
  blackphase/.style={very thick,black,postaction={decorate},
    decoration={markings,mark=at position 0.60 with {\arrow{Stealth[length=2mm]}}}},
  greenfield/.style={figgreen!85!black,thin},
  grayfield/.style={gray!65,thin},
  garr/.style={-{Stealth[length=1.8mm]},figgreen!85!black,thin},
  barr/.style={-{Stealth[length=1.8mm]},figblue!85!black,thin},
  state/.style={circle,draw=black,fill=white,inner sep=1.4pt},
  regionfill/.style={pattern=dots,pattern color=figred!45}
}

\title{Existence of Large Boundary Layer Solutions for the Outflow Problem of Full Compressible Navier-Stokes Equations}
\author{
	Tianle Wang\thanks{
		Institute of Applied Mathematics, Academy of Mathematics and Systems Science, Chinese Academy of Sciences, Beijing 100190, P. R. China and School of Mathematical Sciences, University of Chinese Academy of Sciences, Beijing 100049, P. R. China (\href{mailto:wangtianle25@mails.ucas.ac.cn}{wangtianle25@mails.ucas.ac.cn}).
	},
	Yi Wang\thanks{
		State Key Laboratory of Mathematical Sciences and Institute of Applied Mathematics, Academy of Mathematics and Systems Science, Chinese Academy of Sciences, Beijing 100190, P. R. China and School of Mathematical Sciences, University of Chinese Academy of Sciences, Beijing 100049, P. R. China (\href{mailto:wangyi@amss.ac.cn}{wangyi@amss.ac.cn}). The
work of Yi Wang is partially supported by NSFC grants (Grant No.s  12288201 and 12421001)
and CAS Project for Young Scientists in Basic Research, Grant No. YSBR-031. 
	},
	and Qiuyang Yu\thanks{
		Institute of Applied Mathematics, Academy of Mathematics and Systems Science, Chinese Academy of Sciences, Beijing 100190, P. R. China and School of Mathematical Sciences, University of Chinese Academy of Sciences, Beijing 100049, P. R. China (\href{mailto:yuqiuyang@amss.ac.cn}{yuqiuyang@amss.ac.cn}).
	}
}
\date{}

\begin{document}

\maketitle

\begin{abstract}

We investigate the existence and non-existence of large-amplitude boundary layer solutions to the outflow problem for the one-dimensional full compressible Navier-Stokes equations on the half-line $\mathbb{R}_+$.
Through a delicate global phase-plane analysis, we substantially extend the small-amplitude boundary layer solutions obtained via the center-manifold approach to the large-amplitude regime. Based on the sign of $M_+-1$ (with $M_+$ denoting the Mach number at the right end state) and the Prandtl-number-related parameter $\zeta$, we provide a complete characterization of the existence and non-existence of large-amplitude boundary layer solutions for this half-space outflow problem.
In sharp contrast to the corresponding inflow problem, boundary layer solutions to the outflow problem are permitted to be non-monotone, and their existence region in the phase plane can be an open set.

\end{abstract}


\tableofcontents

\section{Introduction}

In this paper, we consider the one-dimensional full (or non-isentropic) compressible Navier-Stokes-Fourier equations on the half-line $\Rp:=(0,+\infty)$:
\begin{equation}
\begin{cases}
\rho_t+(\rho u)_x=0,\qquad\qquad\qquad   t\in\Rp,\ x\in\Rp,\\
(\rho u)_t+(\rho u^2+p)_x=\mu u_{xx},\\
\displaystyle
\left[\rho\left(e+\dfrac{u^2}{2}\right)\right]_t
+\left[\rho u\left(e+\dfrac{u^2}{2}\right)+pu\right]_x
=\kappa\theta_{xx}+\mu(uu_x)_x.
\end{cases}
\label{eq:1.1}
\end{equation}
Here $\rho=\rho(t,x)>0$, $u=u(t,x)$, and $\theta=\theta(t,x)>0$ respectively denote the density, velocity, and absolute temperature of the fluid. The thermodynamic variables $p=p(\rho,\theta)$ and $e=e(\rho,\theta)$ denote the pressure and specific internal energy, while the viscosity coefficient $\mu$ and thermal conductivity coefficient $\kappa$ are both assumed to be positive constants. For an ideal polytropic gas, the corresponding constitutive relations take the form:
\begin{equation}
p=R\rho\theta=A\rho^\gamma\exp\!\left(\dfrac{\gamma-1}{R}s\right),
\qquad
e=\dfrac{R\theta}{\gamma-1}+\mathrm{const},
\label{eq:1.2}
\end{equation}
where $s=s(t,x)$ stands for the entropy, $\gamma>1$ is the adiabatic exponent, and $A,R>0$ are universal gas constants.

We equip the system \eqref{eq:1.1} with the initial data
\begin{equation}
(\rho,u,\theta)(0,x)=(\rho_0,u_0,\theta_0)(x)
\longrightarrow(\rho_+,u_+,\theta_+),\quad
{\rm as}\quad x\to+\infty,
\label{eq:1.3}
\end{equation}
which satisfies the positivity constraints
\[
\inf_{x\in\Rp}\rho_0(x)>0,
\quad
\inf_{x\in\Rp}\theta_0(x)>0,
\quad
\rho_+>0,
\quad
\theta_+>0.
\]
Meanwhile, we impose the following outflow boundary condition at the spatial origin:
\begin{equation}
u(t,0)=u_-<0,
\qquad
\theta(t,0)=\theta_->0.
\label{eq:1.4}
\end{equation}
Note that for the
outflow problem \eqref{eq:1.1}-\eqref{eq:1.4} and the impermeable wall problem‌ (i.e., $u_-=0$), the density cannot be prescribed at the boundary $x=0$, due to the well-posedness of the hyperbolic mass equation $\eqref{eq:1.1}_1$. For the inflow problem with $u_->0$, however, an additional boundary condition for density must be imposed at $x=0$, see Matsumura \cite{Matsumura2001}. 

One of the central problems in the mathematical theory of viscous compressible fluids is to investigate the large-time asymptotic behavior of solutions to the governing equations.
For the Cauchy problem on the whole line $\mathbb{R}$, the large-time behavior of solutions to the compressible Navier-Stokes equations \eqref{eq:1.1} is closely related to the Riemann solutions of the corresponding Euler system, which consists of shocks, rarefaction waves, and contact discontinuities. Since the pioneering work of Il'in and Oleinik \cite{IlinOleinik1960}, the stability of these basic wave patterns and their combinations has been extensively studied by using energy estimates, spectral methods, Green's function techniques, and contraction arguments, as documented in representative works including \cite{Liu1985,LiuXin1988,SzepessyXin1993,Goodman1986,FreistuhlerSerre1998,MasciaZumbrun2004,HuangXinYang2008,HuangLiMatsumura2010,LiuZeng2015}. Very recently, Kang, Vasseur, and Wang \cite{KangVasseurWang2023,KangVasseurWang2025} successfully proved the time-asymptotic stability of generic Riemann solutions for both barotropic and full compressible Navier-Stokes systems by using $a$-contraction method.

On the half-line $\mathbb{R}_+$, the presence of a physical boundary at $x=0$ may generate a new type of stationary wave, known as a boundary layer solution to the compressible Navier-Stokes equations.
For the barotropic compressible Navier-Stokes system, Matsumura \cite{Matsumura2001} first established a rigorous criterion that characterizes the emergence of such boundary layer solutions, which can be naturally extended to the full non-isentropic Navier-Stokes-Fourier system \eqref{eq:1.1}, as subsequently demonstrated in \cite{HuangLiShi2010}. 
Intuitively, the necessity of a boundary layer solution can be determined by examining the compatibility between the boundary data and the far-field state: one first verifies whether the Riemann solution of the corresponding Euler system, with the possible left state given by the prescribed boundary data and the right state given by the far-field asymptotic state, satisfies the imposed boundary condition at $x=0$.
 If the compatibility condition holds, no extra near-boundary adjustment is required. Otherwise, a non-trivial boundary layer solution must emerge to resolve the mismatch, smoothly connecting the prescribed boundary state to the far-field state by the stationary solution. For the half space problem \eqref{eq:1.1}, such large-time boundary layer phenomena can arise in both the inflow and outflow regimes.

The existence and stability of boundary layer solutions for viscous fluids have attracted considerable attention over the past decades, leading to a rich body of literature. In the context of the inflow problem, Matsumura and Nishihara \cite{MatsumuraNishihara2001} first studied the existence and stability of boundary layer solutions to the isentropic compressible Navier-Stokes system, and then Huang, Matsumura, and Shi \cite{HuangMatsumuraShi2003} proved the time-asymptotic stability of the superposition of a boundary layer solution and a viscous shock by using anti-derivative techniques and shift arguments. Then Qin and Wang \cite{QinWang2009,QinWang2011} proved the existence of small-amplitude boundary layer solutions to the inflow problem of full compressible Navier-Stokes-Fourier equations in Lagrangian coordinates by the center-manifold approach and the time-asymptotic stability of composite waves including boundary layer solutions, rarefaction waves, and contact discontinuities in both subsonic and transonic scenarios. On the other hand, Nakamura and Nishibata \cite{NakamuraNishibata2011} established the existence and stability of the inflow boundary layer solution in Eulerian coordinates. Very recently, Han, Kang, Kim, Kim, and Oh \cite{HKKKO} proved the time-asymptotic stability of the superposition of the degenerate boundary layer solution, the rarefaction wave, and the viscous shock wave to the inflow barotropic compressible Navier-Stokes equations.

For the outflow problem, Kawashima, Nishibata, and Zhu \cite{KawashimaNishibataZhu2003} first established the existence and stability of boundary layer solutions for the one-dimensional barotropic  compressible Navier-Stokes system. Then Kagei and Kawashima \cite{KageiKawashima2006} investigated the multidimensional stability of such outflow boundary layer solutions; Nakamura, Nishibata, and Yuge \cite{NakamuraNishibataYuge2007} studied its large-time stability and convergence rates in one-dimensional half-line; Kawashima, Nakamura, Nishibata, and Zhu \cite{KawashimaNakamuraNishibataZhu2010} further proved the existence of small-amplitude outflow boundary layers by the center-manifold theory and
analyzed their time-asymptotic stability for the full compressible Navier-Stokes-Fourier system; and then Qin \cite{Qin2011} and Wan, Wang, and Zou \cite{WanWangZou2016} proved the stability of outflow boundary layers under large initial perturbations, while requiring the small boundary layer amplitude. Additionally, Chen, Hong, and Shi \cite{ChenHongShi2019,ChenHongShi2021} generalized the existence and stability analysis for the small-amplitude outflow boundary layer solutions of the viscous fluids for the ideal polytropic gas to more general equations of state.

For the spectral stability of boundary layer solutions to the viscous conservation laws including the Navier-Stokes equations \eqref{eq:1.1}, one can refer to the recent progress by Costanzino, Humpherys, Nguyen, and
Zumbrun \cite{CHNZ}, Nguyen, and Zumbrun \cite{NK} and Gu\`{e}s, M\'{e}tivier, Williams, and Zumbrun \cite{GMWZ} and the references therein. 

As outlined above, the boundary layer solutions to the half-space inflow/outflow problem are constructed via the stationary solutions of the governing equations. For the barotropic compressible Navier–Stokes equations, direct integration reduces the stationary boundary layer system to one algebraic relation for the mass (density) equation and a single scalar first-order ODE for the momentum (velocity) equation. This fully decoupled scalar ODE structure enables a rigorous existence of boundary layer solutions with arbitrary amplitude.

In sharp contrast, for the full compressible Navier–Stokes–Fourier system \eqref{eq:1.1}, the mass (density) equation still reduces to an algebraic relation upon integration, exactly as in the barotropic setting. However, the momentum (velocity) and energy (temperature) equations remain  strongly coupled, and their direct integration yields a first-order planar autonomous ordinary differential system. Crucially, standard center-manifold theory only guarantees the existence of boundary layer solutions to this planar system in a sufficiently small neighborhood of the far-field equilibrium, thereby restricting the constructed profiles to a small-amplitude regime. To extend these local, small-amplitude results to the global, large-amplitude framework,
a delicate global phase-plane analysis of the resulting first-order planar autonomous ordinary differential system is indispensable.

Regarding the corresponding inflow problem, Wang, Yang, and Yu \cite{WangYangYu2025} recently established the necessary and sufficient conditions for the existence of large-amplitude boundary layer solutions via a global geometric phase-plane analysis, which was inspired in part by Gilbarg for the existence of viscous shock profiles \cite{Gilbarg1951}. Crucially, their work demonstrates that the local invariant manifolds can be globally continued, and that the set of physically admissible inflow data can be precisely characterized by a family of well-defined monotone phase-plane curves.

However, for the outflow problem \eqref{eq:1.1}–\eqref{eq:1.4}, the global geometry of the profile trajectories in the phase-plane differs substantially from that of the inflow problem, as these trajectories are no longer guaranteed to be monotone and may become unbounded. Moreover, the supersonic case introduces an additional finite equilibrium, with its stable separatrices delineating the boundary of the basin attracted to the prescribed far-field state. Thus, classifying the outflow profiles requires considerations well beyond the local type of the far-field equilibrium or a straightforward adaptation of the inflow argument.

The present paper aims to establish the necessary and sufficient conditions for the existence of large-amplitude outflow boundary layer solutions associated with \eqref{eq:1.1}–\eqref{eq:1.4} by a global phase-plane analysis.
Our analysis hinges on the signs of two key quantities:   $M_+-1$ (with $M_+$ denoting
the Mach number at the right end state) and the Prandtl-number-related parameter $\zeta$. The far-field Mach number $M_+$, defined as
\begin{equation}
\Mplus:=\dfrac{|u_+|}{\sqrt{R\gamma\theta_+}}>0,
\label{eq:1.5}
\end{equation}
classifies the right far field into subsonic, transonic, or supersonic cases depending on the sign of $M_+-1$.

Our classification reveals several distinctive features absent from the inflow problem. In the subsonic case ($M_+-1<0$), the admissible boundary states form two distinct curves emanating from the far-field equilibrium. In the transonic case ($M_+-1=0$), these states comprise two boundary curves along with an open region bounded by them and the phase-plane boundary $\theta=0$.  In the supersonic case ($M_+-1>0$), the aforementioned auxiliary saddle equilibrium possesses two stable separatrices; these curves delineate an open basin of attraction, the interior of which consists of boundary layer trajectories that asymptotically approach the prescribed far-field state. The separatrices themselves tend toward the auxiliary saddle equilibrium instead, and thus do not constitute admissible outflow data for the prescribed far-field state. The resulting boundary layer solutions, while potentially non-monotone, are shown to be piecewise monotone and confined entirely within the physical region.

Another key quantity in the global continuation argument is the parameter
\begin{align}
	\zeta:=\dfrac{2}{\mu} - \dfrac{R}{\kappa(\gamma-1)} + 2\sqrt{\dfrac{R}{\kappa\mu}},
	\label{eq:2.2}
\end{align}
which stems from the comparison estimates used to track how a characteristic curve continues away from an equilibrium.
In terms of the Prandtl number
\[
\Pr=\frac{\mu c_p}{\kappa},
\qquad
c_p=\frac{\gamma R}{\gamma-1},
\]
one has
\begin{equation}\label{Prandtl-related}
\mu\zeta
=2-\frac{\Pr}{\gamma}
+2\sqrt{\frac{(\gamma-1)\Pr}{\gamma}}.
\end{equation}
Thus, for a fixed adiabatic exponent $\gamma$, the sign of $\zeta$ reflects the relative strength of momentum and thermal diffusions, satisfying
\begin{equation}\label{sign-z}
\zeta\geq0
\quad\Longleftrightarrow\quad
0<\Pr\leq 2\gamma\bigl(\gamma+\sqrt{\gamma^2-1}\bigr).
\end{equation}
This sign determines whether the relevant curve remains monotone or develops a turning point. By synthesizing this comparison argument with invariant regions, the classification of finite equilibria, and the global continuation of separatrices, we provide a complete characterization of the admissible boundary data in subsonic, transonic and supersonic cases, respectively.

The rest of the paper is organized as follows. In Section~2, we formulate the stationary outflow problem and state the main theorem. In Section~3, we reduce the stationary equations to a planar autonomous system and prove the main theorem by a global phase-plane analysis. Section~4 discusses the time-asymptotic stability problem for the large-amplitude boundary layer solutions and concludes the paper.

\section{Main Result}
\label{sec:main}

For the outflow problem \eqref{eq:1.1}-\eqref{eq:1.4}, the boundary layer solution $(\rho^{\mB},u^{\mB},\theta^{\mB})(x)$ satisfies
\begin{equation}
\begin{cases}
	(\rho^{\mB} u^{\mB})' = 0, \\[6pt]
	\bigl(\rho^{\mB} (u^{\mB})^2 + p^{\mB}\bigr)' = \mu (u^{\mB})'', \\[6pt]
	\left[\rho^{\mB} u^{\mB} \left(\dfrac{R\theta^{\mB}}{\gamma-1} + \dfrac{(u^{\mB})^2}{2}\right) + p^{\mB} u^{\mB}\right]' = \bigl(\kappa (\theta^{\mB})' + \mu u^{\mB} (u^{\mB})'\bigr)', \\[6pt]
	(u^{\mB}, \theta^{\mB})(0) = (u_-, \theta_-), \\[6pt]
	(\rho^{\mB}, u^{\mB}, \theta^{\mB})(+\infty) = (\rho_+, u_+, \theta_+),
\end{cases}
\quad'=\frac{d}{dx}, \quad x \in \mathbb{R}_+,
\label{eq:2.1}
\end{equation}
where $u_- < 0$, $\rho_+>0$, $\theta_{\pm}>0$, and $p^{\mB}=R\rho^{\mB}\theta^{\mB}$. The stationary boundary layer problem is therefore reduced to the existence of solutions to the ODE system \eqref{eq:2.1}. 

Our main result is stated as follows. 

\begin{theorem}
\label{thm:main}
For the outflow boundary layer problem \eqref{eq:2.1} with $\rho_+>0$, $\theta_\pm>0$, and $u_-<0$, we have
\begin{itemize}
	\item If $u_+\ge 0$, then there is no boundary layer solution to \eqref{eq:2.1}.
	
	\item If $u_+<0$, then the existence and non-existence of large boundary layer solutions to \eqref{eq:2.1} can be divided into the following three cases:
	
	\begin{enumerate}[label=(\alph*)]
		\item \textbf{Subsonic case} ($0<\Mplus<1$). There exist two distinct curves $\Gamma_1$ and $\Gamma_2$ (see Figure \ref{fig1}) in the phase $(u^{\mB},\theta^{\mB})$-plane such that the following holds.
		\begin{enumerate}[label=(\roman*)]
			\item The boundary layer solution $(\rho^{\mB},u^{\mB},\theta^{\mB})(x)$ to \eqref{eq:2.1} exists if and only if
				$(u_-,\theta_-)\in \Gamma_1\cup\Gamma_2$.
			\item $\Gamma_1$ is monotone in the $(u^{\mB},\theta^{\mB})$-plane, while $\Gamma_2$ is monotone when $\zeta\leq0$ and non-monotone when $\zeta>0$;
			\item Both $\Gamma_1$ and $\Gamma_2$ are tangent to the straight line
			\[
			\rho_+u_+^2(u^{\mB}-u_+)+\Mplus^2\gamma\kappa\left(\dfrac{R\rho_+u_+}{\kappa(\gamma-1)}-\lambda_2\right)(\theta^{\mB}-\theta_+)=0,
			\]
			at $(u_+,\theta_+)$, where $\lambda_2<0$ is the negative eigenvalue of the matrix $A$ (to be defined in \eqref{eq:3.15}).
			\item There exist constants $C_n>0$ and $c>0$ such that the boundary layer solution $(\rho^{\mB},u^{\mB},\theta^{\mB})(x)$ satisfies
			\begin{equation}
				\left|\dfrac{\mathrm{d}^n}{\mathrm{d}x^n}(\rho^{\mB}-\rho_+,u^{\mB}-u_+,\theta^{\mB}-\theta_+)\right|
				\le C_n\me^{-cx},
				\qquad n=0,1,2,\dots.
				\label{eq:2.3}
			\end{equation}
		\end{enumerate}
		
		\item \textbf{Transonic case} ($\Mplus=1$). There exist two distinct phase-plane curves $\Gamma_3$, $\Gamma_4$, and an open region $\Sigma_1$ bounded below by the phase-plane boundary $\theta^{\mB}=0$ and above by $\Gamma_3$ and $\Gamma_4$ (see Figure \ref{fig2}), such that the following holds.
		\begin{enumerate}[label=(\roman*)]
			\item $\Gamma_3$ is monotone in the $(u^{\mB},\theta^{\mB})$-plane, while $\Gamma_4$ is monotone when $\zeta\leq0$ and non-monotone when $\zeta>0$.
			\item Both $\Gamma_3$ and $\Gamma_4$ are tangent to the straight line
			\[
			\mu u_+(u^{\mB}-u_+)-\kappa(\gamma-1)(\theta^{\mB}-\theta_+)=0
			\]
	        at $(u_+,\theta_+)$.
	        \item The boundary layer solution $(\rho^{\mB}, u^{\mB}, \theta^{\mB})(x)$ to \eqref{eq:2.1} exists if and only if
	        $(u_-,\theta_-)\in \Sigma_1\cup\Gamma_3\cup\Gamma_4$. Moreover, the boundary layer solution $(\rho^{\mB}, u^{\mB}, \theta^{\mB})(x)$ is piecewise monotone.
			\item There exist constants $C_n>0$ such that the boundary layer solution $(\rho^{\mB}, u^{\mB}, \theta^{\mB})(x)$ satisfies
			\begin{equation}
				\left|\dfrac{\mathrm{d}^n}{\mathrm{d}x^n}(\rho^{\mB}-\rho_+,u^{\mB}-u_+,\theta^{\mB}-\theta_+)\right|
				\le \dfrac{C_n}{(1+x)^{n+1}},
				\qquad n=0,1,2,\dots.
				\label{eq:2.4}
			\end{equation}
		\end{enumerate}
		
		\item \textbf{Supersonic case} ($\Mplus>1$). There exist two distinct phase-plane curves $\Gamma_5$, $\Gamma_6$, and an open region $\Sigma_2$ bounded below by the phase-plane boundary $\theta^{\mB}=0$ and above by $\Gamma_5$ and $\Gamma_6$ (see Figure \ref{fig3}), such that the following holds.
		\begin{enumerate}[label=(\roman*)]
			
			\item $\Gamma_5$ is monotone in the $(u^{\mB},\theta^{\mB})$-plane, while $\Gamma_6$ is monotone when $\zeta\leq0$ and non-monotone when $\zeta>0$.
			\item Both $\Gamma_5$ and $\Gamma_6$ are tangent to the straight line
			\[
			\begin{aligned}
			&\left(\dfrac{(\Mplus^2\gamma^2-3\Mplus^2\gamma+3\gamma-1)\rho_+u_+}{(\Mplus^2(\gamma-1)+2)\mu\gamma}-\lambda_4\right)
			\bigl(u^{\mB}-(1+\alpha_1)u_+\bigr)\\
			&\quad +\dfrac{R\rho_+\Mplus^2(\gamma+1)}{(\Mplus^2(\gamma-1)+2)\mu}
			\bigl(\theta^{\mB}-(1+\alpha_2)\theta_+\bigr)=0
			\end{aligned}
			\]
			 at the point $S_*((1+\alpha_1)u_+,(1+\alpha_2)\theta_+)$, where $\alpha_1,\alpha_2$ are given in \eqref{eq:3.38} and $\lambda_4<0$ is the negative eigenvalue of $\widetilde A$ (to be defined in \eqref{eq:3.39}).
			 \item The boundary layer solution $(\rho^{\mB}, u^{\mB}, \theta^{\mB})(x)$ to \eqref{eq:2.1} exists if and only if
			 $(u_-,\theta_-)\in \Sigma_2\setminus\{(u_+,\theta_+)\}$. Moreover, the boundary layer solution $(\rho^{\mB}, u^{\mB}, \theta^{\mB})(x)$ is piecewise monotone.
			\item There exist constants $C_n>0$ and $c>0$ such that the boundary layer solution $(\rho^{\mB}, u^{\mB}, \theta^{\mB})(x)$ satisfies
			\begin{equation}
				\left|\dfrac{\mathrm{d}^n}{\mathrm{d}x^n}(\rho^{\mB}-\rho_+,u^{\mB}-u_+,\theta^{\mB}-\theta_+)\right|
				\le C_n\me^{-cx},
				\qquad n=0,1,2,\dots.
				\label{eq:2.5}
			\end{equation}
		\end{enumerate}
	\end{enumerate}
\end{itemize}
\end{theorem}

\begin{remark}
 Theorem~\ref{thm:main} classifies all the admissible outflow boundary data with large amplitude for the boundary layer problem \eqref{eq:2.1} in the subsonic, transonic and supersonic cases respectively.
\end{remark}



\begin{remark}
    To demonstrate the scope of $\zeta$ in most physically relevant cases, Table~\ref{tab:physical-zeta} lists six common gases at $300{\rm K}$ and at $1 {\rm atm}$ with all physical parameters expressed in SI units \cite{WangFeng1982}. Direct substitution of these values into \eqref{eq:2.2} yields $\zeta>0$ for all six gases. Therefore, for these gases, the cases in Theorem~\ref{thm:main} with $\zeta>0$ represent the most physically relevant ones.
	
\end{remark}

\begin{table}[htbp]
\centering
\scriptsize
\setlength{\tabcolsep}{3.2pt}
\renewcommand{\arraystretch}{1.05}
\caption{Physical parameters and the corresponding values of $\zeta$ at $300\,\mathrm{K}$ and $1\,\mathrm{atm}$ in SI units.}
\label{tab:physical-zeta}
\begin{tabular}{lccccc}
\toprule
Gas & $\mu$ & $\kappa$ & $R$ & $\gamma$ & $\zeta$ \\
 & $(\mathrm{Pa}\cdot\mathrm{s})$ & $(\mathrm{W}\,\mathrm{m}^{-1}\,\mathrm{K}^{-1})$ & $(\mathrm{J}\,\mathrm{kg}^{-1}\,\mathrm{K}^{-1})$ & & \\
\midrule
Air & $1.85\times10^{-5}$ & $2.62\times10^{-2}$ & $287.0$ & $1.40$ & $1.294\times10^{5}$ \\
Nitrogen ($\mathrm{N}_2$) & $1.77\times10^{-5}$ & $2.59\times10^{-2}$ & $296.8$ & $1.40$ & $1.352\times10^{5}$ \\
Oxygen ($\mathrm{O}_2$) & $2.07\times10^{-5}$ & $2.66\times10^{-2}$ & $259.8$ & $1.40$ & $1.156\times10^{5}$ \\
Carbon dioxide ($\mathrm{CO}_2$) & $1.50\times10^{-5}$ & $1.66\times10^{-2}$ & $188.9$ & $1.29$ & $1.492\times10^{5}$ \\
Hydrogen ($\mathrm{H}_2$) & $8.90\times10^{-6}$ & $1.869\times10^{-1}$ & $4124.0$ & $1.41$ & $2.705\times10^{5}$ \\
Methane ($\mathrm{CH}_4$) & $1.11\times10^{-5}$ & $3.43\times10^{-2}$ & $518.3$ & $1.31$ & $2.052\times10^{5}$ \\
\bottomrule
\end{tabular}
\end{table}

\newpage

\begin{figure}[htbp]
	\centering
	\begin{subfigure}{0.4\textwidth}
		\centering
		\includegraphics[width=\linewidth]{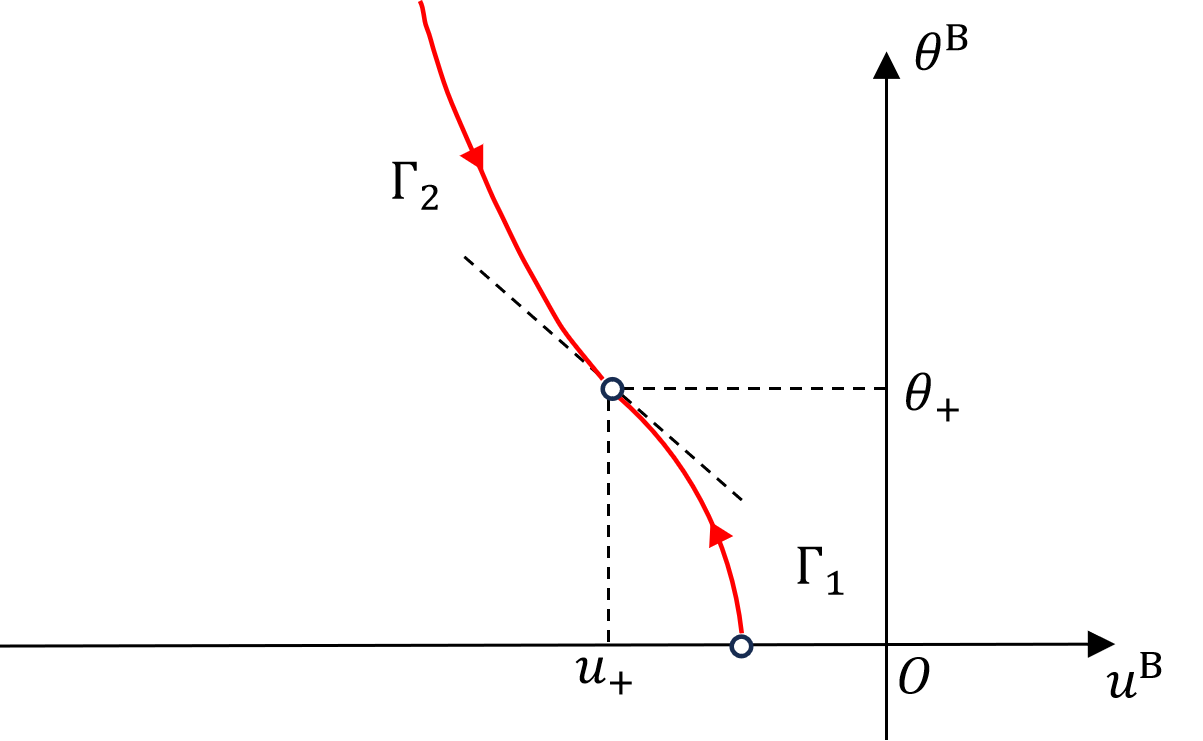}
		\caption{$\zeta\leq0$}
		\label{fig1(1)}
	\end{subfigure}
	\hfill 
	\begin{subfigure}{0.4\textwidth}
		\centering
		\includegraphics[width=\linewidth]{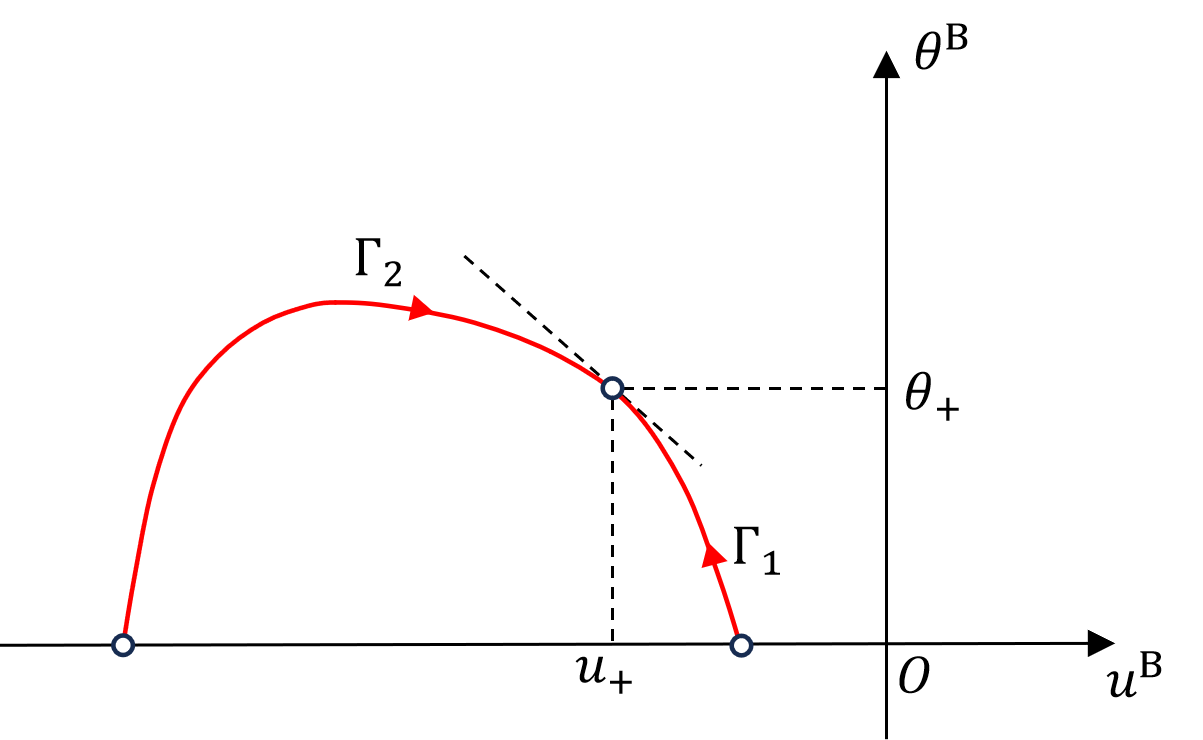}
		\caption{$\zeta>0$}
		\label{fig1(2)}
	\end{subfigure}
	\caption{Subsonic case: $0<\Mplus<1$. Admissible boundary data set: $\Gamma_1\cup\Gamma_2$.}
	\label{fig1}
\end{figure}

\begin{figure}[H]
	\centering
	\begin{subfigure}{0.4\textwidth}
		\centering
		\includegraphics[width=\linewidth]{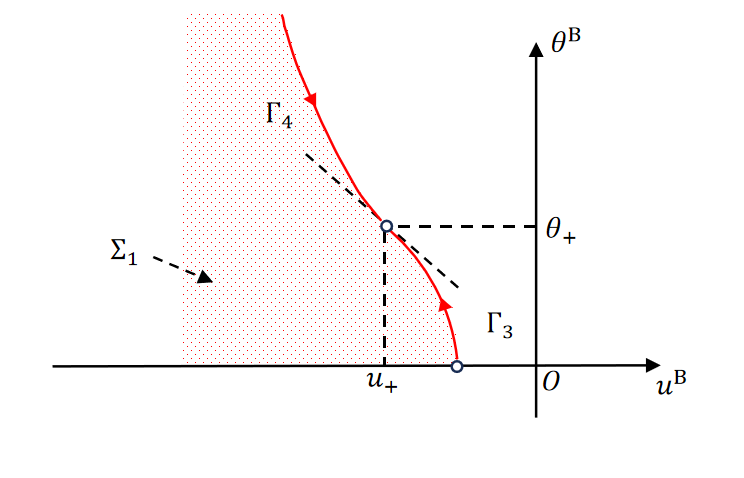}
		\caption{$\zeta\leq0$}
		\label{fig2(1)}
	\end{subfigure}
	\hfill 
	\begin{subfigure}{0.4\textwidth}
		\centering
		\includegraphics[width=\linewidth]{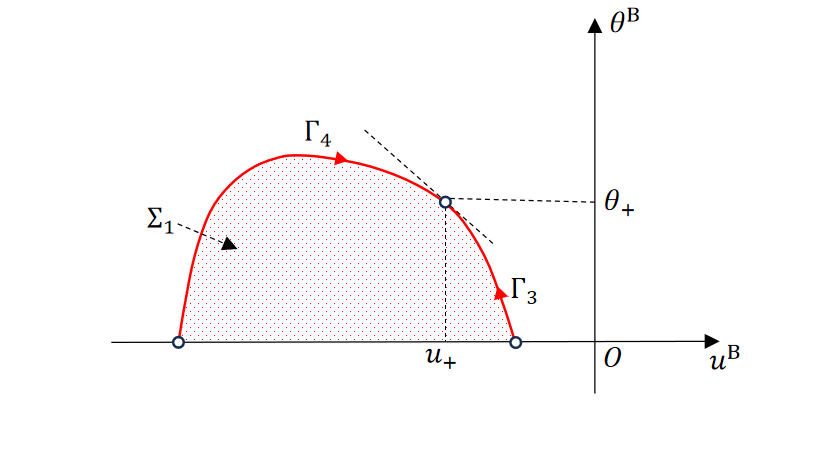}
		\caption{$\zeta>0$}
		\label{fig2(2)}
	\end{subfigure}
	\caption{Transonic case: $\Mplus=1$. Admissible boundary data set: $\Sigma_1\cup\Gamma_3\cup\Gamma_4$.}
	\label{fig2}
\end{figure}

\begin{figure}[H]
	\centering
	\begin{subfigure}{0.4\textwidth}
		\centering
		\includegraphics[width=\linewidth]{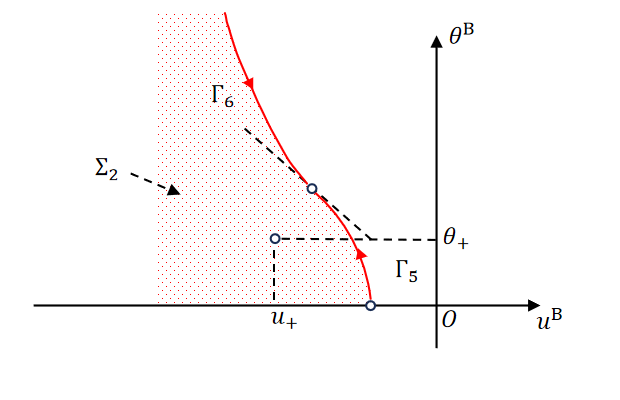}
		\caption{$\zeta\leq0$}
		\label{fig3(1)}
	\end{subfigure}
	\hfill 
	\begin{subfigure}{0.4\textwidth}
		\centering
		\includegraphics[width=\linewidth]{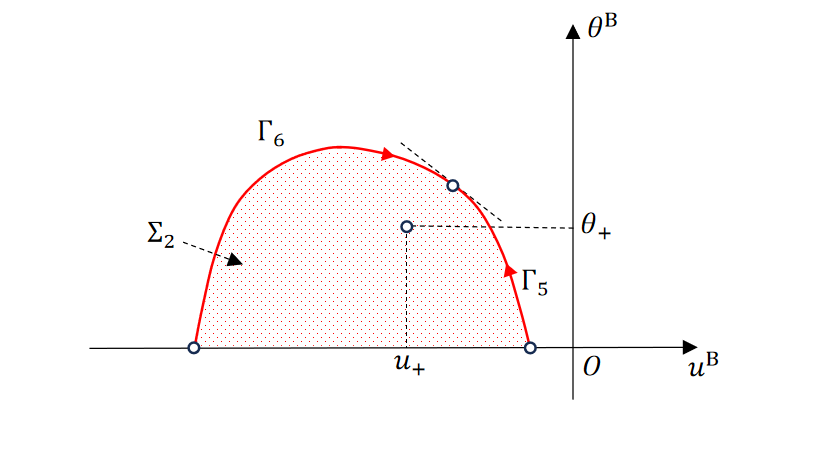}
		\caption{$\zeta>0$}
		\label{fig3(2)}
	\end{subfigure}
	\caption{Supersonic case: $\Mplus>1$. Admissible boundary data set: $\Sigma_2\setminus\{(u_+,\theta_+)\}$.}
	\label{fig3}
\end{figure}

\section{Proof of the Main Result}
\subsection{Reduction to a planar autonomous system}

Integrating the stationary boundary layer equations \eqref{eq:2.1} over $[x,+\infty)$ and using the far-field condition at $x=+\infty$, we obtain the first integrals
\begin{equation}
\begin{cases}
\rho^{\mB} u^{\mB}=\rho_+u_+,\\[1mm]
\rho^{\mB} (u^{\mB})^2+R\rho^{\mB}\theta^{\mB}-\mu (u^{\mB})'=\rho_+u_+^2+R\rho_+\theta_+,\\[1mm]
\rho^{\mB} u^{\mB}\left(\dfrac{R\theta^{\mB}}{\gamma-1}+\dfrac{(u^{\mB})^2}{2}\right)+R\rho^{\mB}\theta^{\mB} u^{\mB}
-\kappa(\theta^{\mB})'-\mu u^{\mB}(u^{\mB})' \\
\qquad =\rho_+u_+\left(\dfrac{R\theta_+}{\gamma-1}+\dfrac{u_+^2}{2}\right)+R\rho_+\theta_+u_+ ,\\[6pt]
(u^{\mB}, \theta^{\mB})(0) = (u_-, \theta_-), \\
(\rho^{\mB}, u^{\mB}, \theta^{\mB})(+\infty) = (\rho_+, u_+, \theta_+).
\end{cases}
\label{eq:3.4}
\end{equation}
The first identity in \eqref{eq:3.4} is the conservation of the stationary mass flux. Since we are dealing with the outflow boundary condition $u_-<0$, a physically admissible boundary layer solution must satisfy $u^{\mB}(x)<0$ and $\rho^{\mB}(x)>0$ for all $x\ge0$. Therefore the constant flux in \eqref{eq:3.4} must be negative. If $u_+\ge0$, this is impossible; hence no outflow boundary layer solution exists in this case. In the rest of the proof we always assume $u_+<0$.

By the first identity in \eqref{eq:3.4}, the density is recovered from $u^{\mB}$ as
\begin{equation}
\rho^{\mB}(x)=\dfrac{\rho_+u_+}{u^{\mB}(x)}.
\label{eq:3.5}
\end{equation} 
Thus $\rho^{\mB}(0)=\dfrac{\rho_+u_+}{u^{\mB}(0)}>0$. Set
\begin{equation}
	\bar{u}^{\mB}:=u^{\mB}-u_+,
	\qquad
	\bar\theta^{\mB}:=\theta^{\mB}-\theta_+ .
	\label{eq:3.1}
\end{equation}
Then the far-field state becomes the origin $O(0,0)$ in the $(\bar{u}^{\mB},\bar\theta^{\mB})$-plane. Curves and regions in this shifted plane will be denoted by primes, for instance $\Gamma_i'$ and $\Sigma_i'$, while their translations in the original $(u^{\mB},\theta^{\mB})$-plane are denoted by $\Gamma_i$ and $\Sigma_i$. The physical constraints are
\begin{equation}
	\bar\theta^{\mB}> -\theta_+,
	\qquad
	\bar{u}^{\mB}<-u_+,
	\label{eq:3.2}
\end{equation}
where the first inequality is the positivity of the absolute temperature and the second follows from $u^{\mB}<0$ in the outflow case. We denote the physical lower boundary by
\begin{equation}
	L_0:=\{(\bar{u}^{\mB},\bar\theta^{\mB}):\bar\theta^{\mB}=-\theta_+\}.
	\label{eq:3.3}
\end{equation}
The reduced vector field is singular on the straight line
\begin{equation}
\mathcal S:=\{(\bar{u}^{\mB},\bar\theta^{\mB}):\bar{u}^{\mB}=-u_+\},
\label{eq:3.6}
\end{equation}
because the density formula \eqref{eq:3.5} blows up there. This line is not an equilibrium curve; it is the boundary of the algebraically reduced phase plane. Consequently all phase-plane solution curves used below are considered in the open physical domain \eqref{eq:3.2}, and a solution curve ceases to be physically admissible once it reaches either $\mathcal S$ or the lower thermodynamic boundary $L_0$.

Substituting \eqref{eq:3.5} into the last two identities in \eqref{eq:3.4}, and using
\[
R\theta_+=\dfrac{u_+^2}{\Mplus^2\gamma},
\]
we obtain a system for $(\bar{u}^{\mB},\bar\theta^{\mB})$:
\begin{equation}
\begin{cases}
(\bar{u}^{\mB})'=H_1(\bar{u}^{\mB},\bar\theta^{\mB}),\\[1mm]
(\bar\theta^{\mB})'=H_2(\bar{u}^{\mB},\bar\theta^{\mB}),
\end{cases}
\qquad
(\bar{u}^{\mB},\bar\theta^{\mB})(+\infty)=(0,0),
\label{eq:3.7}
\end{equation}
where
\begin{align}
H_1(\bar{u}^{\mB},\bar\theta^{\mB})
&:=\dfrac{\rho_+u_+}{\mu}
\left(
\bar{u}^{\mB}+\dfrac{R(\bar\theta^{\mB}+\theta_+)}{\bar{u}^{\mB}+u_+}-\dfrac{u_+}{\Mplus^2\gamma}
\right),
\label{eq:3.8}\\
H_2(\bar{u}^{\mB},\bar\theta^{\mB})
&:=\dfrac{\rho_+u_+^2}{\Mplus^2\gamma\kappa}\bar{u}^{\mB}
+\dfrac{R\rho_+u_+}{\kappa(\gamma-1)}\bar\theta^{\mB}
-\dfrac{\rho_+u_+}{2\kappa}(\bar{u}^{\mB})^2 .
\label{eq:3.9}
\end{align}
Conversely, any non-trivial solution of \eqref{eq:3.7} satisfying the physical constraint \eqref{eq:3.2} defines an outflow boundary layer solution of \eqref{eq:2.1} through \eqref{eq:3.5}. Hence the existence problem for \eqref{eq:2.1} is equivalent to the existence of a physically admissible solution curve of \eqref{eq:3.7} approaching the origin as $x\to+\infty$. 

We next identify the finite equilibria of the planar vector field away from the singular line \eqref{eq:3.6}. Solving $H_1=H_2=0$ we obtain two equilibria
\begin{equation}
O(0,0),
\qquad
P(\alpha_1u_+,\alpha_2\theta_+),
\label{eq:3.10}
\end{equation}
where
\begin{equation}
\alpha_1:=\dfrac{2(1-\Mplus^2)}{\Mplus^2(\gamma+1)},
\qquad
\alpha_2:=\dfrac{2(\Mplus^2-1)(\Mplus^2\gamma+1)(\gamma-1)}{\Mplus^2(\gamma+1)^2}.
\label{eq:3.11}
\end{equation}
In the original $(u^{\mB},\theta^{\mB})$-plane, these equilibria are
\begin{equation}
S_+=(u_+,\theta_+),
\qquad
S_*=\bigl((1+\alpha_1)u_+,(1+\alpha_2)\theta_+\bigr).
\label{eq:3.13}
\end{equation}
When $\Mplus=1$, one has $\alpha_1=\alpha_2=0$, and therefore $S_*$ coincides with $S_+$. Since the boundary layer solution curve is required to satisfy $(\bar{u}^{\mB},\bar\theta^{\mB})(+\infty)=(0,0)$, the local type of $O$ and the global continuation of the invariant-manifold curves issuing from the stable eigendirections will determine the admissible boundary data.

The linearization of \eqref{eq:3.7} at $O$ is
\begin{equation*}
\binom{(\bar{u}^{\mB})'}{(\bar\theta^{\mB})'}
=A\binom{\bar{u}^{\mB}}{\bar\theta^{\mB}}+O\left(|(\bar{u}^{\mB},\bar\theta^{\mB})|^2\right),
\label{eq:3.14}
\end{equation*}
where
\begin{align}
A=
\begin{pmatrix}
\dfrac{(\Mplus^2\gamma-1)\rho_+u_+}{\Mplus^2\gamma\mu} & \dfrac{R\rho_+}{\mu}\\[2mm]
\dfrac{\rho_+u_+^2}{\Mplus^2\gamma\kappa} & \dfrac{R\rho_+u_+}{\kappa(\gamma-1)}
\end{pmatrix}.
\end{align}
A direct calculation gives
\begin{align}
\det A=\dfrac{R\rho_+^2u_+^2(\Mplus^2-1)}{\Mplus^2\mu\kappa(\gamma-1)},
\qquad
\mathrm{tr} A=\rho_+u_+\left(\dfrac{M_+^2\gamma-1}{M_+^2\gamma\mu}+\dfrac{R}{\kappa(\gamma-1)}\right).
\label{eq:3.15}
\end{align}
Therefore the origin is a saddle for $0<\Mplus<1$, a degenerate equilibrium for $\Mplus=1$, and a stable node for $\Mplus>1$.

Let
\begin{equation}
h_1(\bar{u}^{\mB}):=-\dfrac{1}{R}\bar{u}^{\mB}\left(\bar{u}^{\mB}+\dfrac{(\Mplus^2\gamma-1)u_+}{\Mplus^2\gamma}\right),
\qquad
h_2(\bar{u}^{\mB}):=\dfrac{\gamma-1}{2R}\bar{u}^{\mB}\left(\bar{u}^{\mB}-\dfrac{2u_+}{\Mplus^2\gamma}\right).
\label{eq:3.16}
\end{equation}
Here $\bar\theta^{\mB}=h_1(\bar{u}^{\mB})$ defines the nullcline 
 along which $(\bar{u}^{\mB})'=0$, while $\bar\theta^{\mB}=h_2(\bar{u}^{\mB})$ defines the nullcline for which $(\bar\theta^{\mB})'=0$. Let $L_1$ and $L_2$ be the graphs of $h_1$ and $h_2$. The direction of the vector field on these curves is the main tool in the global continuation argument below.

\subsection{Subsonic case: $0<\Mplus<1$}

When $0<\Mplus<1$, \eqref{eq:3.15} gives $\det A<0$. Hence $O$ is a saddle. Let $\lambda_1>0>\lambda_2$ be the two eigenvalues of $A$. The stable eigenvector corresponding to $\lambda_2$ has negative slope. More precisely, an eigenvector associated with $\lambda_i$ can be chosen as
\begin{equation*}
\left(1,-\dfrac{\rho_+u_+^2}{\Mplus^2\gamma\kappa}\left(\dfrac{R\rho_+u_+}{\kappa(\gamma-1)}-\lambda_i\right)^{-1}\right),
\qquad i=1,2,
\label{eq:3.17}
\end{equation*}
and the sign relation
\begin{equation*}
\operatorname{sgn}\left(\dfrac{(\Mplus^2\gamma-1)\rho_+u_+}{\Mplus^2\gamma\mu}-\lambda_2\right)
=\operatorname{sgn}(\lambda_1-\lambda_2)>0
\label{eq:3.18}
\end{equation*}
shows that the stable direction has negative slope. Therefore, the one-dimensional stable manifold of $O$ has two branches, approaching $O$ from the fourth and second quadrants, respectively. In particular, they are tangent at $O$ to the straight line
\begin{equation}
\tau_1':=\left\{(\bar{u}^{\mB},\bar \theta^{\mB})\;\biggl|\;\rho_+u_+^2\bar{u}^{\mB}+\Mplus^2\gamma\kappa\left(\dfrac{R\rho_+u_+}{\kappa(\gamma-1)}-\lambda_2\right)\bar\theta^{\mB}=0\right\}.
\label{eq:3.19}
\end{equation}
We denote these two branches by $\Gamma_1'$ and $\Gamma_2'$.

First consider the curve $\Gamma_1'$. Define
\begin{equation*}
l_1:=\{(\bar{u}^{\mB},\bar\theta^{\mB})\;|\;\bar\theta^{\mB}=h_1(\bar{u}^{\mB}),\ 0\le \bar{u}^{\mB}\le -u_+\},
\end{equation*}
and let Region I be the region enclosed by $\bar{u}^{\mB}=0$, $L_0$, and $l_1$; see Figure~\ref{fig:3.1}. We continue the stable curve from the far-field equilibrium $O$ in the direction of decreasing $x$. Since $H_1=0$ and $H_2\ge0$ on $l_1$, the vector field is vertical and points upward there. Thus the backward continuation of $\Gamma_1'$ cannot leave Region I through $l_1$. Moreover, in Region I one has
\[
H_1<0,\qquad H_2>0.
\]
In other words,
\begin{equation*}
\dfrac{\mathrm{d}\bar\theta^{\mB}}{\mathrm{d}\bar{u}^{\mB}}<0,
\end{equation*}
so that the curve is monotone in the phase plane. Since Region I contains no finite equilibrium point of \eqref{eq:3.7}, the curve cannot terminate in the interior as $x$ decreases. It must therefore meet the lower thermodynamic boundary $L_0$ at a unique point, denoted by $Z_1$. After this point the extension would enter the non-physical region $\theta^{\mB}<0$, so we keep only the segment of $\Gamma_1'$ between $O$ and $Z_1$.

Conversely, every point $Y\in\Gamma_1'\setminus\{O,Z_1\}$ generates the same orbit with the forward orientation. By the boundary-vector-field argument above, this orbit remains in Region I, stays in the physical domain, and approaches $O$ as $x\to+\infty$. It therefore defines an admissible boundary layer solution.

\begin{figure}[htbp]
\centering
\includegraphics[width=0.3\textwidth]{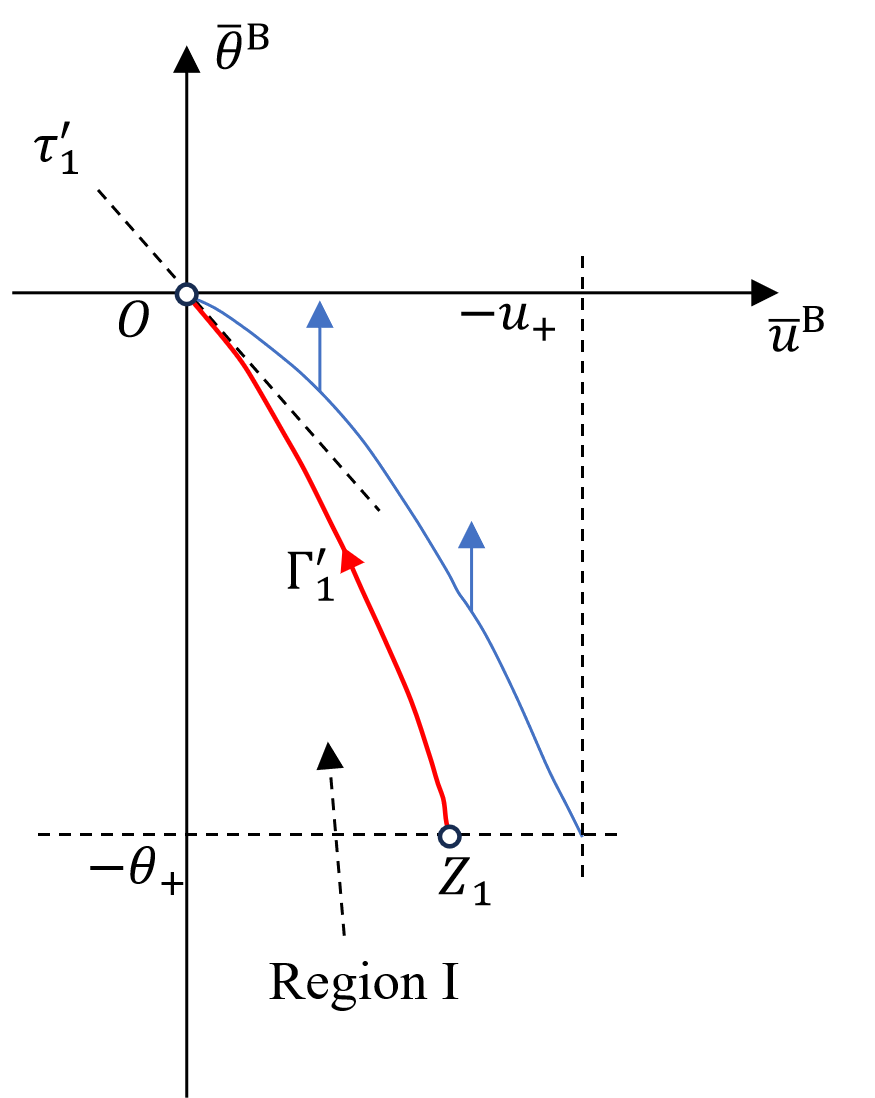}
\caption{Continuation of $\Gamma_1'$ in the subsonic outflow case.}
\label{fig:3.1}
\end{figure}

Next consider $\Gamma_2'$. Near $O$ it lies in the upper-left region
\begin{equation}
\mathrm{II}:=\{(\bar{u}^{\mB},\bar\theta^{\mB})\;|\;\bar{u}^{\mB}<0,\ \bar\theta^{\mB}> \max\{h_1(\bar{u}^{\mB}),h_2(\bar{u}^{\mB})\}\}.
\label{eq:3.20}
\end{equation}
If it leaves Region II through $\bar\theta^{\mB}=h_2(\bar{u}^{\mB})$, it enters
\begin{equation*}
\mathrm{III}:=\{(\bar{u}^{\mB},\bar\theta^{\mB})\;|\;h_1(\bar{u}^{\mB})\le \bar\theta^{\mB}\le h_2(\bar{u}^{\mB}),\ \bar{u}^{\mB}\le \alpha_1u_+,\ \bar\theta^{\mB}>-\theta_+\},
\label{eq:3.21}
\end{equation*}
where
\begin{equation*}
\alpha_1:=\dfrac{2(1-\Mplus^2)}{\Mplus^2(\gamma+1)}.
\label{eq:3.22}
\end{equation*}
In Region II one has $(\bar{u}^{\mB})'>0$, $(\bar\theta^{\mB})'<0$ and $\dfrac{\mathrm{d}\bar\theta^{\mB}}{\mathrm{d}\bar{u}^{\mB}}<0$, while in Region III one has $(\bar{u}^{\mB})'>0$, $(\bar\theta^{\mB})'>0$ and $\dfrac{\mathrm{d}\bar\theta^{\mB}}{\mathrm{d}\bar{u}^{\mB}}>0$. Thus the curve is monotone within each region and can leave only through the indicated nullcline or the physical boundary, unless it remains unbounded in the same region. The possible continuations of $\Gamma_2'$ are therefore the following five situations, illustrated in Figures~\ref{fig:3.2}--\ref{fig:3.6}:
\begin{enumerate}
\item[(A)] $\Gamma_2'$ never leaves Region II and tends to a vertical asymptote;
\item[(B)] $\Gamma_2'$ never leaves Region II and is unbounded as $x\to-\infty$;
\item[(C)] $\Gamma_2'$ enters Region III but tends to a horizontal asymptote;
\item[(D)] $\Gamma_2'$ enters Region III and first meets $\bar\theta^{\mB}=h_1(\bar{u}^{\mB})$;
\item[(E)] $\Gamma_2'$ enters Region III and first meets $L_0$.
\end{enumerate}
\begin{figure}[htbp]
	\centering
	\includegraphics[width=0.7\textwidth]{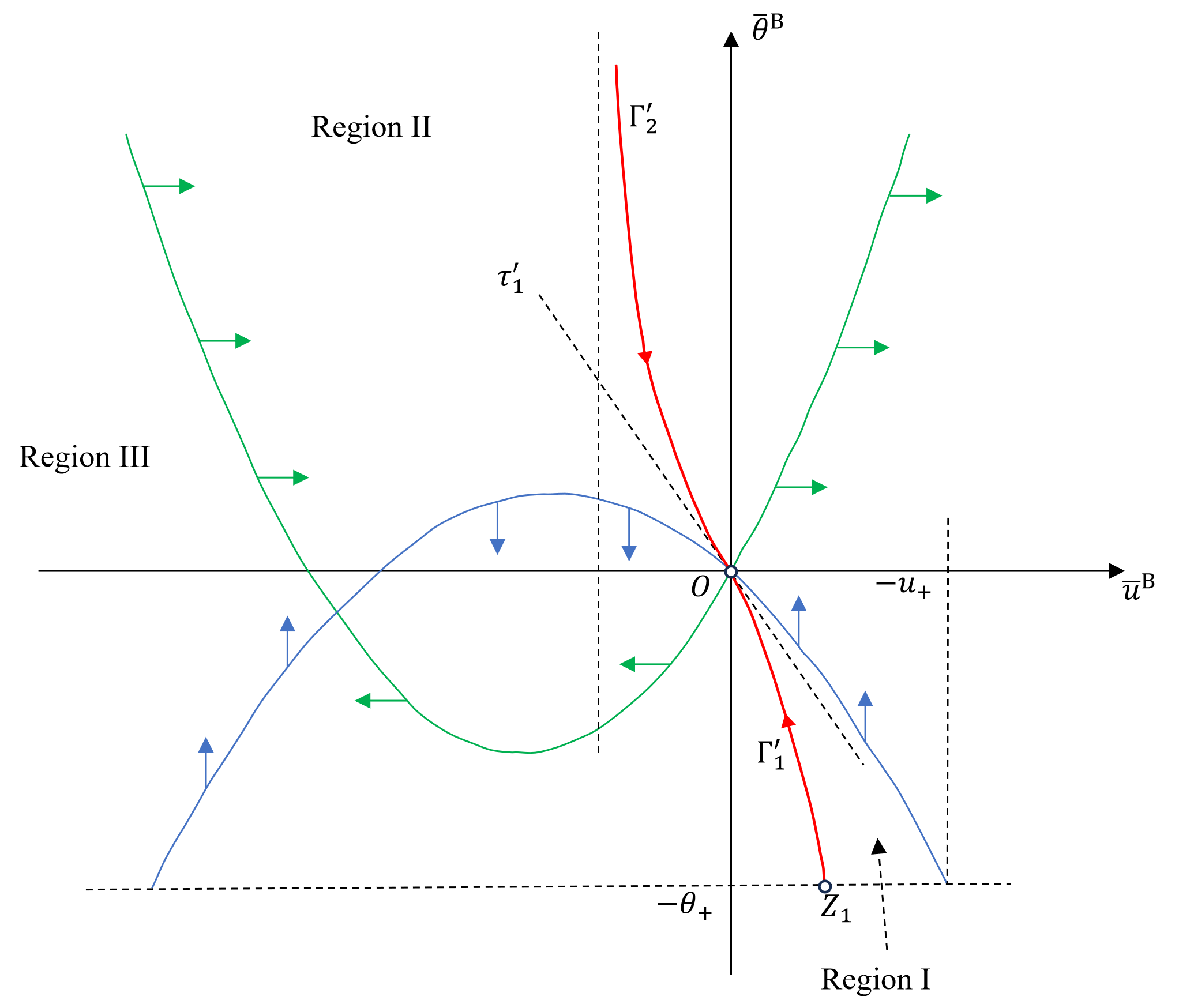}
	\caption{Situation (A).}
	\label{fig:3.2}
\end{figure}

\begin{figure}[htbp]
	\centering
	\includegraphics[width=0.7\textwidth]{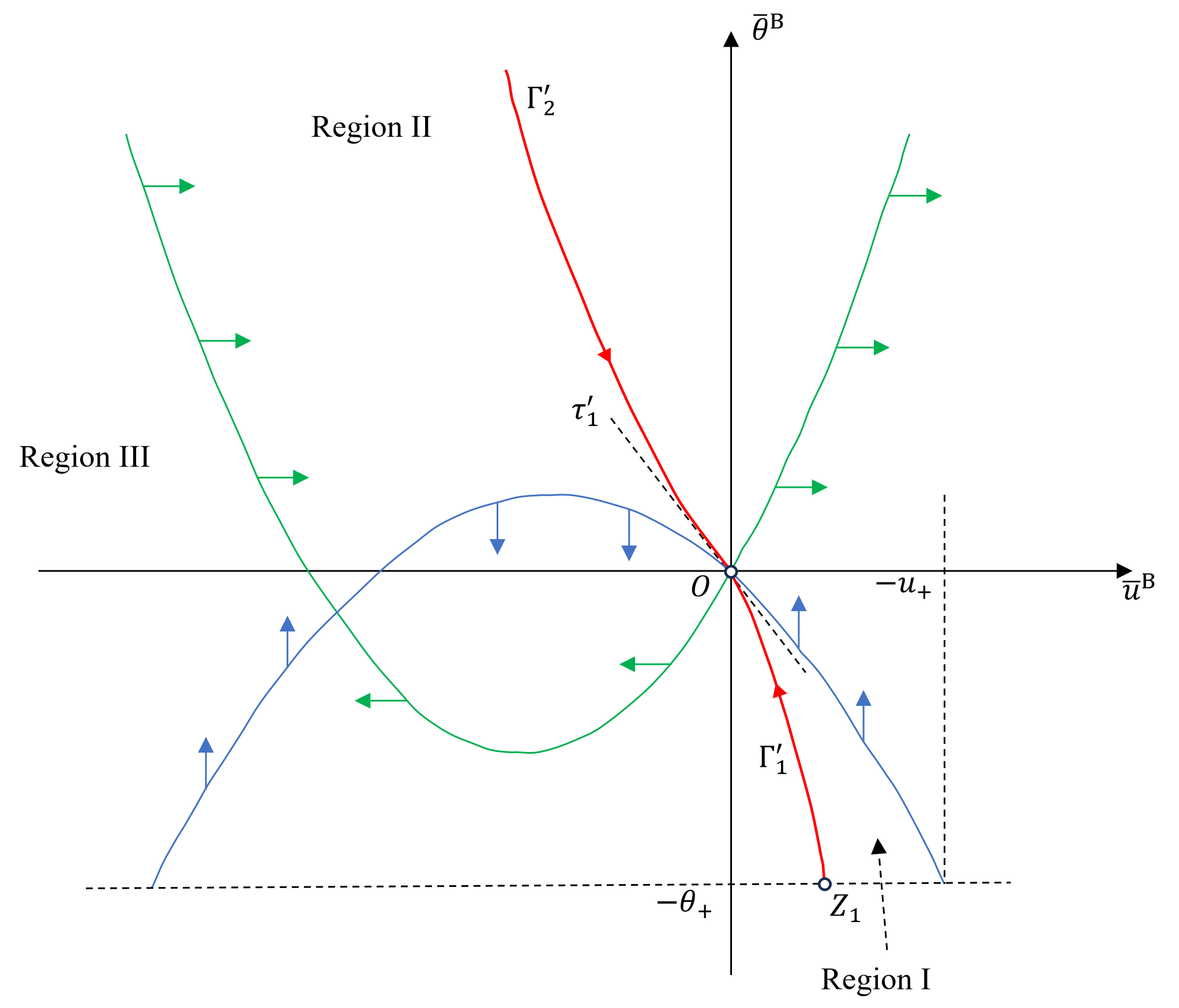}
	\caption{Situation (B).}
	\label{fig:3.3}
\end{figure}

\begin{figure}[htbp]
	\centering
	\includegraphics[width=0.8\textwidth]{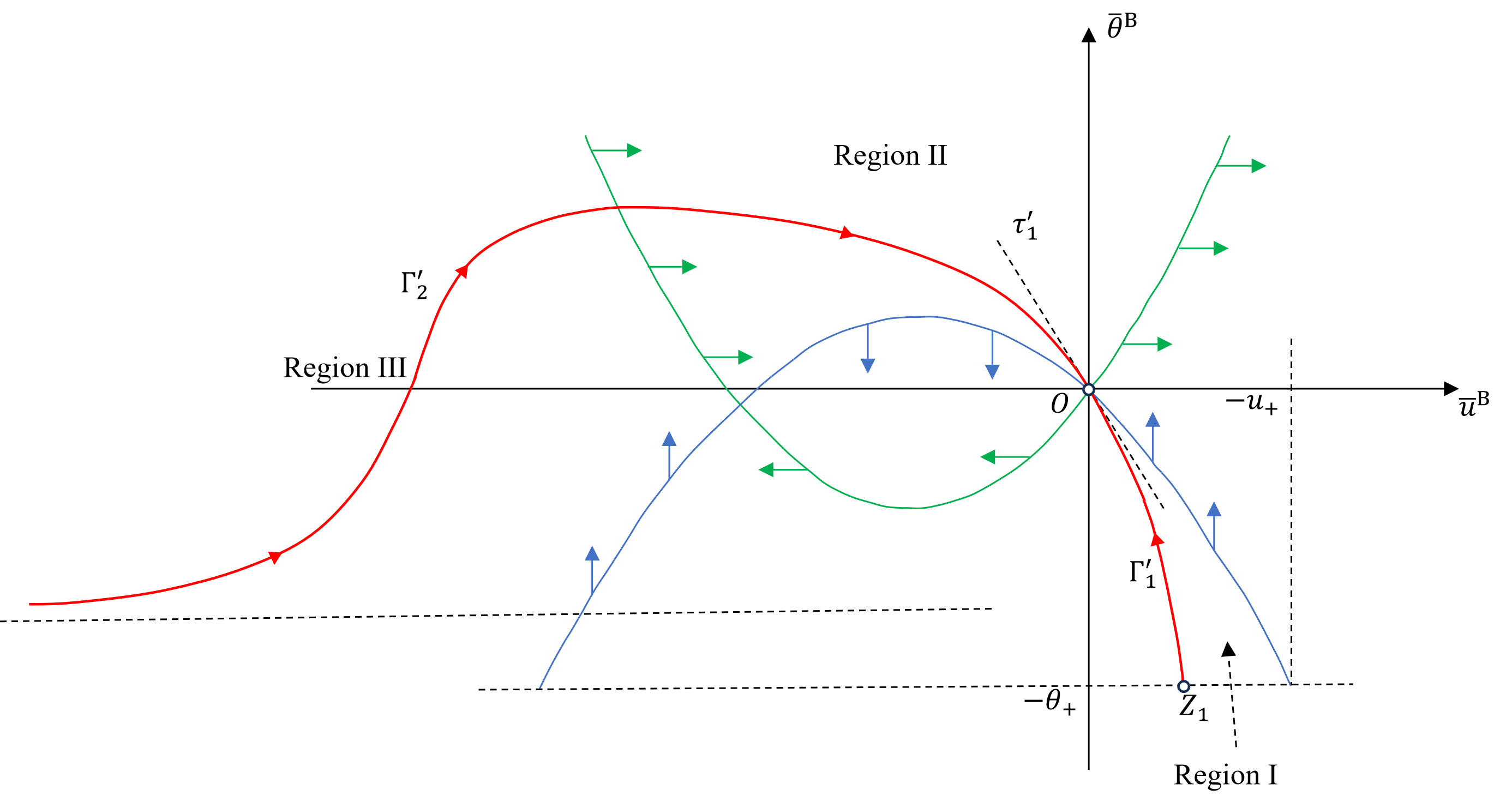}
	\caption{Situation (C).}
	\label{fig:3.4}
\end{figure}

\begin{figure}[htbp]
	\centering
	\includegraphics[width=0.8\textwidth]{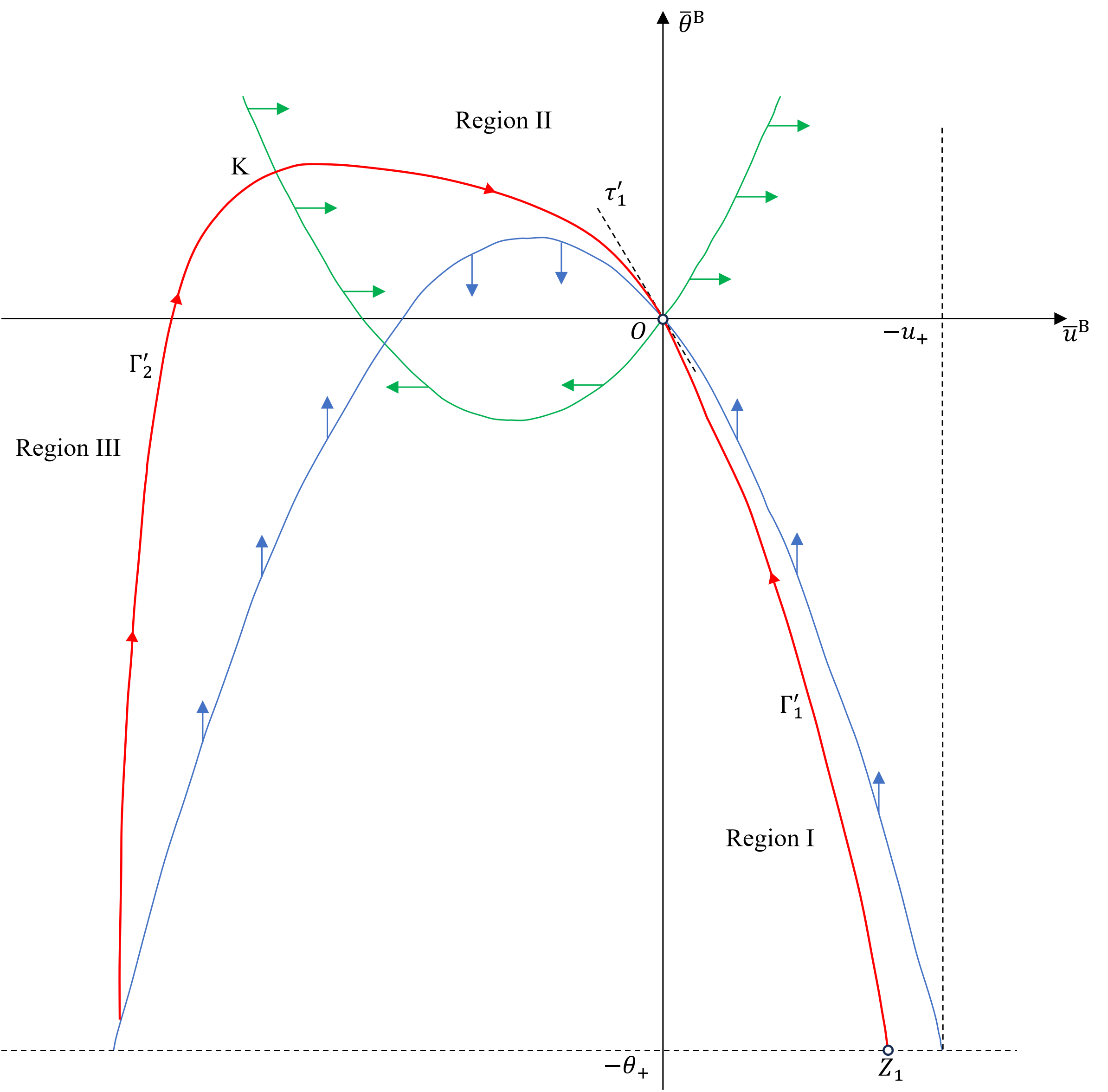}
	\caption{Situation (D).}
	\label{fig:3.5}
\end{figure}

\begin{figure}[htbp]
	\centering
	\includegraphics[width=0.8\textwidth]{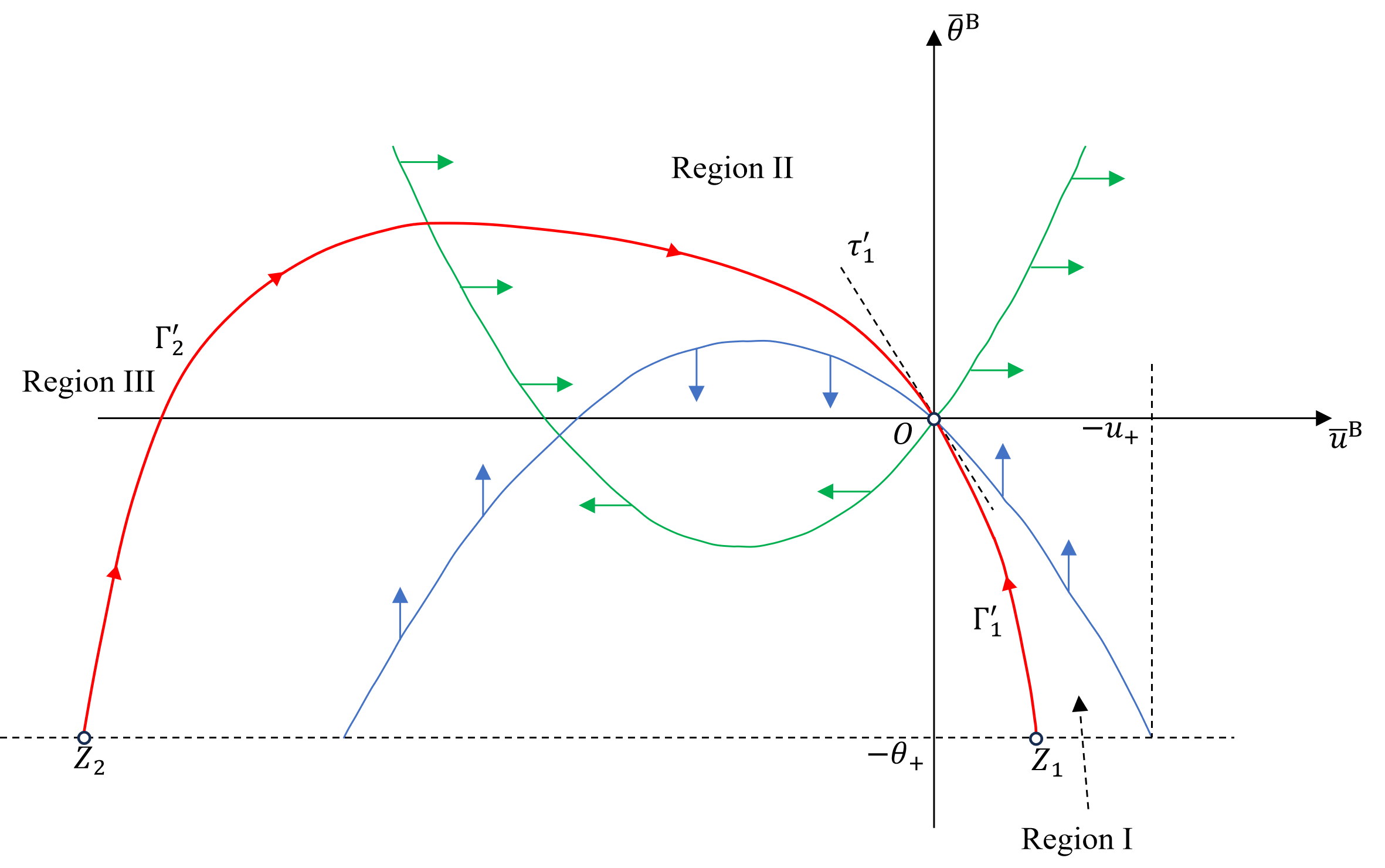}
	\caption{Situation (E).}
	\label{fig:3.6}
\end{figure}

We now exclude situations (A)(C)(D) in detail. By \eqref{eq:3.7} and \eqref{eq:3.16} we have
\begin{align}
	\dfrac{\mathrm{d}\bar \theta^{\mB}}{\mathrm{d}\bar{u}^{\mB}}
	&=
	\dfrac{
		\left(\bar{u}^{\mB} + u_+ \right)
		\displaystyle\left(
		\dfrac{R}{\kappa(\gamma-1)} \bar \theta^{\mB}
		- \dfrac{1}{2\kappa} (\bar{u}^{\mB})^2
		+ \dfrac{u_+}{M_+^2 \gamma \kappa} \bar{u}^{\mB}\right)
	}{\displaystyle\dfrac{(M_+^2 \gamma - 1)u_+}{M_+^2 \gamma \mu} \bar{u}^{\mB}+ \dfrac{R}{\mu} \bar \theta^{\mB}
		+ \dfrac{1}{\mu} (\bar{u}^{\mB})^2}
	\label{eq:3.28}
	\\&=
	\dfrac{\mu(\bar{u}^{\mB}+u_+)(\bar \theta^{\mB}-h_2(\bar{u}^{\mB}))}
	{\kappa(\gamma-1)(\bar \theta^{\mB}-h_1(\bar{u}^{\mB}))}.
	\label{eq:3.29}
\end{align}

Suppose first that (A) occurs. Then the curve stays in Region II and, when it is continued backward, approaches a vertical asymptote. Thus there is a finite number $C<0$ such that
\begin{equation*}
\bar{u}^{\mB}\to C,
\qquad
\bar\theta^{\mB}\to+\infty .
\end{equation*}
Since the curve remains away from the singular line $\bar{u}^{\mB}=-u_+$, the denominator in \eqref{eq:3.29} is dominated by $\dfrac{R}{\mu}\bar\theta^{\mB}$, while the numerator is dominated by $\dfrac{R(\bar{u}^{\mB}+u_+)}{\kappa(\gamma-1)}\bar\theta^{\mB}$. Consequently,
\begin{equation}
\dfrac{\mathrm{d}\bar\theta^{\mB}}{\mathrm{d}\bar{u}^{\mB}}
\longrightarrow
\dfrac{\mu(C+u_+)}{\kappa(\gamma-1)},
\qquad x\to-\infty,
\label{eq:3.23}
\end{equation}
which is finite. This contradicts the geometric meaning of a vertical asymptote, for which the slope $\dfrac{\mathrm{d}\bar\theta^{\mB}}{\mathrm{d}\bar{u}^{\mB}}$ must be unbounded. Hence (A) is impossible.

Situation (C) is ruled out by the dual argument. In this case the curve has already entered Region III and approaches a horizontal asymptote, so that
\begin{equation*}
\bar{u}^{\mB}\to-\infty,
\qquad
\bar\theta^{\mB}\to C
\end{equation*}
for some finite constant $C>-\theta_+$. Substituting this into \eqref{eq:3.29}, the dominant terms are
\begin{equation*}
(\bar{u}^{\mB}+u_+)
\left(-\dfrac{1}{2\kappa}(\bar{u}^{\mB})^2\right)
\quad\hbox{in the numerator},
\qquad
\dfrac{1}{\mu}(\bar{u}^{\mB})^2
\quad\hbox{in the denominator}.
\end{equation*}
Therefore
\begin{equation}
\dfrac{\mathrm{d}\bar\theta^{\mB}}{\mathrm{d}\bar{u}^{\mB}}
=-\dfrac{\mu}{2\kappa}\bar{u}^{\mB}+o(1)\bar{u}^{\mB}\to+\infty,
\qquad \bar{u}^{\mB}\to-\infty.
\label{eq:3.24}
\end{equation}
This is incompatible with a horizontal asymptote, along which the slope must tend to zero. Hence (C) is impossible.

It remains to exclude (D). Suppose (D) holds, let $K$ be the first point where $\Gamma_2'$ leaves Region II and enters Region III. Then $K$ lies on the nullcline $\bar\theta^{\mB}=h_2(\bar{u}^{\mB})$. Since $K$ is in the second quadrant, its abscissa is strictly to the left of the negative zero of $h_2$, namely
\begin{equation}
\bar{u}^{\mB}=\dfrac{2u_+}{\Mplus^2\gamma}<0.
\label{eq:3.25}
\end{equation}
Inside Region III we have $(\bar{u}^{\mB})'>0$ and $(\bar\theta^{\mB})'>0$. Thus, before the solution curve can leave Region III through another boundary, its abscissa must still belong to the interval to the left of \eqref{eq:3.25}. If (D) were true, the first exit point from Region III would lie on $\bar\theta^{\mB}=h_1(\bar{u}^{\mB})$ with
\begin{equation*}
\bar{u}^{\mB}\le \dfrac{2u_+}{\Mplus^2\gamma}.
\end{equation*}
However, on this interval the curve $L_1$ is already below the physical boundary $L_0$. Indeed,
\begin{equation*}
h_1(\bar{u}^{\mB})+\theta_+
=-\dfrac{1}{R}\left(\bar{u}^{\mB}-\dfrac{u_+}{\Mplus^2\gamma}\right)(\bar{u}^{\mB}+u_+)<0,
\qquad
\bar{u}^{\mB}\le\dfrac{2u_+}{\Mplus^2\gamma},
\end{equation*}
because $u_+<0$. In particular,
\begin{equation}
h_1\left(\dfrac{2u_+}{\Mplus^2\gamma}\right)
=-\dfrac{2u_+^2(\Mplus^2\gamma+1)}{R\Mplus^4\gamma^2}
=-\theta_+\dfrac{2(\Mplus^2\gamma+1)}{\Mplus^2\gamma}< -\theta_+.
\label{eq:3.26}
\end{equation}
Therefore, before $\Gamma_2'$ can touch the boundary $\bar\theta^{\mB}=h_1(\bar{u}^{\mB})$ of Region III, it must first cross $L_0$. This contradicts the defining property of (D).

We now distinguish the situations (B) and (E). We first define
\begin{align}
	f(q)=\dfrac{\displaystyle \dfrac{R}{\kappa(\gamma-1)}q - \dfrac{1}{2\kappa}}{\displaystyle \dfrac{2R}{\mu}q + \dfrac{2}{\mu}}.
\end{align}
A direct calculation gives
\begin{equation*}
f(q)=q
\quad\Longleftrightarrow\quad
\dfrac{2R}{\mu}q^2
+\left(\dfrac{2}{\mu}-\dfrac{R}{\kappa(\gamma-1)}\right)q
+\dfrac{1}{2\kappa}=0.
\end{equation*}
This equation has a positive root if and only if $\zeta\leq0$, where $\zeta$ is defined in \eqref{eq:2.2}. We next show that this condition is equivalent to situation (B).

We introduce new variables
\begin{align}
	\eta=(\bar{u}^{\mB})^2, \qquad \alpha=\dfrac{\bar \theta^{\mB}}{(\bar{u}^{\mB})^2}.
\end{align}
The following comparison estimate is the key step:
\begin{lemma}\label{lem:3.1}
	For $(\bar{u}^{\mB},\bar \theta^{\mB})\in$ II, it holds
	\begin{equation}
		\dfrac{\mathrm{d}\alpha}{\mathrm{d}\eta}>\dfrac{f(\alpha)-\alpha}{\eta}.
	\end{equation}
\end{lemma}
\begin{proof}
	Since $(\bar{u}^{\mB},\bar \theta^{\mB})\in$ II, we have $\bar{u}^{\mB}<0$. By \eqref{eq:3.20} we have
	\begin{align*}
		0 &< \bar \theta^{\mB} + \dfrac{(\bar{u}^{\mB})^2}{R} + \dfrac{(M_+^2 \gamma - 1) u_+}{RM_+^2 \gamma} \bar{u}^{\mB} \\
		&< \bar \theta^{\mB} \dfrac{\bar{u}^{\mB} + u_+}{\bar{u}^{\mB}} + \dfrac{(\bar{u}^{\mB})^2}{R} + \dfrac{u_+}{R} \bar{u}^{\mB} \\
		&= \bigl(\bar{u}^{\mB} + u_+\bigr)\left( \dfrac{\bar \theta^{\mB}}{\bar{u}^{\mB}} + \dfrac{\bar{u}^{\mB}}{R} \right)
	\end{align*}
    and
    \begin{align*}
    	\bar \theta^{\mB} - \dfrac{\gamma - 1}{2R} (\bar{u}^{\mB})^2 + \dfrac{(\gamma - 1) u_+}{RM_+^2 \gamma} \bar{u}^{\mB}>0.
    \end{align*}
    Using \eqref{eq:3.29} we get
    \begin{align*}
    	\dfrac{\mathrm{d}\bar \theta^{\mB}}{\mathrm{d}\bar{u}^{\mB}}
    	&<\dfrac{\mu(\bar{u}^{\mB} + u_+)\displaystyle \left(\bar \theta^{\mB} - \dfrac{\gamma - 1}{2R} (\bar{u}^{\mB})^2 + \dfrac{(\gamma - 1) u_+}{RM_+^2 \gamma} \bar{u}^{\mB}\right)}
    	{\kappa(\gamma-1)\bigl(\bar{u}^{\mB} + u_+\bigr)
    	\displaystyle\left( \dfrac{\bar \theta^{\mB}}{\bar{u}^{\mB}} + \dfrac{\bar{u}^{\mB}}{R} \right)}\\
    	&<\dfrac{\mu\displaystyle \left(\bar \theta^{\mB} - \dfrac{\gamma - 1}{2R} (\bar{u}^{\mB})^2\right)}
    	{\kappa(\gamma-1)\displaystyle\left( \dfrac{\bar \theta^{\mB}}{\bar{u}^{\mB}} + \dfrac{\bar{u}^{\mB}}{R} \right)}\\
    	&=2\bar{u}^{\mB}\dfrac{\displaystyle \dfrac{R}{\kappa(\gamma-1)}\dfrac{\bar \theta^{\mB}}{(\bar{u}^{\mB})^2} - \dfrac{1}{2\kappa}}{\displaystyle \dfrac{2R}{\mu}\dfrac{\bar \theta^{\mB}}{(\bar{u}^{\mB})^2} + \dfrac{2}{\mu}}\\
    	&=2\bar{u}^{\mB}f(\alpha).
    \end{align*}
    Note that
    \begin{align*}
    	\dfrac{\mathrm{d}\bar \theta^{\mB}}{\mathrm{d}\eta}
    	=\dfrac{\mathrm{d}\bar \theta^{\mB}}{2\bar{u}^{\mB}\mathrm{d}\bar{u}^{\mB}},
    	\qquad
    	\dfrac{\mathrm{d}\alpha}{\mathrm{d}\eta}=\dfrac{1}{\eta}\displaystyle\left(
    	\dfrac{\mathrm{d}\bar \theta^{\mB}}{\mathrm{d}\eta}-
    	\dfrac{\bar \theta^{\mB}}{\eta}\right).
    \end{align*}
    Since $\bar{u}^{\mB}<0$, we conclude that
    \begin{equation*}
    	\dfrac{\mathrm{d}\alpha}{\mathrm{d}\eta}>\dfrac{f(\alpha)-\alpha}{\eta}.
    \end{equation*}
\end{proof}

We now prove
\begin{proposition}
	If $f(q)=q$ has positive roots, then (B) occurs.
\end{proposition}
\begin{proof}
	Let $\beta>0$ satisfy $f(\beta)=\beta$, then
	\begin{align*}
		f(\beta) = \dfrac{\displaystyle \dfrac{R}{\kappa(\gamma-1)} \beta - \dfrac{1}{2\kappa}}{\displaystyle \dfrac{2R}{\mu} \beta + \dfrac{2}{\mu}} > 0 \implies \beta > \dfrac{\gamma-1}{2R}.
	\end{align*}
    Note that if $\bar \theta^{\mB}\geq\beta(\bar{u}^{\mB})^2$ for $\bar{u}^{\mB}<0$, then  $\bar \theta^{\mB}>h_2(\bar{u}^{\mB})$; hence, $\Gamma_2'$ never leaves Region II. We will prove $\bar \theta^{\mB}\geq\beta(\bar{u}^{\mB})^2$ for all $\bar{u}^{\mB}<0$, which implies (B) occurs. Since the curve has negative slope near $O$, there exist $C>0$ and $\delta>0$ such that
    \begin{align*}
    	\bar \theta^{\mB}>-C\bar{u}^{\mB}>\beta(\bar{u}^{\mB})^2,\qquad -\sqrt{\delta}<\bar{u}^{\mB}<0.
    \end{align*}
    In other words,
    \begin{align*}
    	\alpha(\eta)>\beta,\qquad 0<\eta<\delta.
    \end{align*}
    Set
    \begin{align*}
    	\mathcal A=\{t>0\;|\;\alpha(\eta)>\beta,\forall\,0<\eta<t\}.
    \end{align*}
    Let $T=\sup \mathcal A$. Suppose $T<\infty$. By continuity, $\alpha(T)=\beta$.
    Since $f(\beta)=\beta$, Lemma \ref{lem:3.1} gives $\left.\dfrac{\mathrm{d}\alpha}{\mathrm{d}\eta}\right|_{\eta=T}>0$. Thus there exists some $\eta\in(0,T)$, such that $\alpha(\eta)<\beta$, which leads to a contradiction with the definition of $\mathcal A$. Consequently, $T=\infty$, which means
    \begin{align*}
    	\dfrac{\bar \theta^{\mB}}{(\bar{u}^{\mB})^2}=\alpha\geq \beta
    \end{align*}
    for all $\bar{u}^{\mB}<0$.
\end{proof}

We proceed to prove
\begin{proposition}
	If $f(q)=q$ has no positive root, then (E) occurs.
\end{proposition}

\begin{proof}
	Note that
	\begin{align*}
		f(0)<0,\qquad \lim\limits_{q\to \infty}(f(q)-q)=-\infty.
	\end{align*}
    Since $f(q)=q$ has no positive root, we have
    \begin{align*}
    	\sup\limits_{q\geq0}\,(f(q)-q)<0.
    \end{align*}
    Consequently, there exists some $\varepsilon_0>0$, such that
    \begin{align}
    	f(q)-q<-\varepsilon_0
    	\label{3.32}
    \end{align}
    for all $q\geq0$. Suppose, to the contrary, that (E) does not occur. Since (A), (C), and (D) have already been excluded, the only remaining possibility is (B); hence,
    \begin{align}
    	\dfrac{\bar \theta^{\mB}}{(\bar{u}^{\mB})^2}=\alpha> 0,\qquad\bar{u}^{\mB}<0.
    	\label{3.33}
    \end{align}
    By \eqref{eq:3.28} we have
    \begin{align*}
    	\dfrac{\mathrm{d}\bar \theta^{\mB}}{\mathrm{d}\eta}
    	&=
    	\dfrac{
    		\left(1+\dfrac{u_+}{\bar{u}^{\mB}}\right)
    		\left(
    		\dfrac{R}{\kappa(\gamma-1)} \dfrac{\bar \theta^{\mB}}{(\bar{u}^{\mB})^2}
    		- \dfrac{1}{2\kappa}
    		+ \dfrac{u_+}{M_+^2 \gamma \kappa} \dfrac{1}{\bar{u}^{\mB}}
    		\right)
    	}{
    		\dfrac{(M_+^2 \gamma -1)u_+}{M_+^2 \gamma \mu} \dfrac{2}{\bar{u}^{\mB}}
    		+ \dfrac{2R}{\mu} \dfrac{\bar \theta^{\mB}}{(\bar{u}^{\mB})^2}
    		+ \dfrac{2}{\mu}
    	}\\
    	&=
    	\dfrac{
    		\big(1+o(1)\big)
    		\left(
    		\dfrac{R}{\kappa(\gamma-1)} \dfrac{\bar \theta^{\mB}}{(\bar{u}^{\mB})^2}
    		- \dfrac{1}{2\kappa}
    		+ o(1)
    		\right)
    	}{
    		o(1)
    		+ \dfrac{2R}{\mu} \dfrac{\bar \theta^{\mB}}{(\bar{u}^{\mB})^2}
    		+ \dfrac{2}{\mu}
    	}\\
    	&=
    	\dfrac{
    		\dfrac{R}{\kappa(\gamma-1)} \dfrac{\bar \theta^{\mB}}{(\bar{u}^{\mB})^2}
    		- \dfrac{1}{2\kappa}
    	}{
    		\dfrac{2R}{\mu} \dfrac{\bar \theta^{\mB}}{(\bar{u}^{\mB})^2}
    		+ \dfrac{2}{\mu}
    	}
    	+ o(1)\\
    	&=f(\alpha)+o(1),\qquad \bar{u}^{\mB}\to -\infty.
    \end{align*}
    Consequently, for any $\varepsilon>0$, there exists some $M>0$, such that
    \begin{align*}
    	\dfrac{\mathrm{d}\bar \theta^{\mB}}{\mathrm{d}\eta}<f(\alpha)+\varepsilon,
    	\qquad \bar{u}^{\mB}<-M.
    \end{align*}
    In other words,
    \begin{align*}
    	\dfrac{\mathrm{d}\alpha}{\mathrm{d}\eta}
    	=\dfrac{1}{\eta}\left(\dfrac{\mathrm{d}\bar \theta^{\mB}}{\mathrm{d}\eta}-\alpha\right)
    	<\dfrac{1}{\eta}\left(f(\alpha)-\alpha+\varepsilon\right),\qquad \eta>M^2.
    \end{align*}
    From \eqref{3.32}, set $\varepsilon=\dfrac{\varepsilon_0}{2}$, we have
    \begin{align*}
    	\dfrac{\mathrm{d}\alpha}{\mathrm{d}\eta}
    	<-\dfrac{\varepsilon_0}{2\eta},\qquad \eta>M^2.
    \end{align*}
    Integrating the above inequality over $[M^2,\eta]$, we deduce
    \begin{align*}
    	\alpha(\eta)-\alpha(M^2)<-\dfrac{\varepsilon_0}{2}
    	\log\dfrac{\eta}{M^2},\qquad \eta>M^2.
    \end{align*}
    This implies $\lim\limits_{\eta\to \infty}\alpha(\eta)=-\infty$, which contradicts \eqref{3.33}. Consequently, (B) is impossible, and hence (E) holds.
\end{proof}

Translating back through \eqref{eq:3.1} gives the curves $\Gamma_1$ and $\Gamma_2$ in the original variables. Their endpoints are excluded under the convention that boundary layer solutions are non-trivial and physically admissible. This completes the proof of part (a) of Theorem~\ref{thm:main}.

\subsection{Transonic case: $\Mplus=1$}

When $\Mplus=1$, one has
\[
\det A=0,
\qquad
\operatorname{tr}A
=\rho_+u_+\left(\dfrac{\gamma-1}{\gamma\mu}+\dfrac{R}{\kappa(\gamma-1)}\right)<0.
\]
Thus $A$ has one zero eigenvalue and one negative eigenvalue $\lambda_3$. Choose the two eigenvectors
\begin{equation*}
e_1=\left(1,-\dfrac{(\gamma-1)u_+}{R\gamma}\right),
\qquad
e_2=\left(1,\dfrac{\mu u_+}{\kappa(\gamma-1)}\right),
\label{eq:3.32}
\end{equation*}
and let $T=(e_1,e_2)$. Under the change of variables
\begin{equation*}
\binom{W_1}{W_2}=T^{-1}\binom{\bar{u}^{\mB}}{\bar\theta^{\mB}},
\label{eq:3.33}
\end{equation*}
the system takes the normal form
\begin{equation}
\begin{cases}
W_1'=g_1(W_1,W_2),\\
W_2'=\lambda_3W_2+g_2(W_1,W_2),
\end{cases}
\qquad \lambda_3<0,
\qquad g_1,g_2=O(|(W_1,W_2)|^2).
\label{eq:3.34}
\end{equation}

\begin{lemma}[Saddle-node criterion for one zero eigenvalue; see \cite{DumortierLlibreArtes2006,ZhangDingHuangDong1985}]
\label{lem:saddle-node-negative}
Consider the planar autonomous system
\begin{equation}
\begin{cases}
\dfrac{\mathrm{d}X}{\mathrm{d}t}=g_1(X,Y),\\[2mm]
\dfrac{\mathrm{d}Y}{\mathrm{d}t}=\lambda Y+g_2(X,Y),
\end{cases}
\qquad \lambda<0.
\label{eq:3.35}
\end{equation}
Assume that $(0,0)$ is an isolated singular point of \eqref{eq:3.35}, and that $g_1,g_2$ are analytic functions of order at least two in a sufficiently small neighborhood $B(0,\delta)$ of $(0,0)$. Let $Y=\phi(X)$, $|X|<\delta$, be the implicit function determined by
\[
\lambda\phi(X)+g_2(X,\phi(X))=0,
\]
and suppose that
\[
\psi(X):=g_1(X,\phi(X))=a_mX^m+o(X^m),
\qquad a_m\ne0,
\qquad m\ge2.
\]
Then the following assertions hold.
\begin{enumerate}[label=\textup{(\arabic*)}]
\item Any non-constant solution curve of \eqref{eq:3.35} which tends to $(0,0)$ as $t\to+\infty$ or $t\to-\infty$ must be tangent to one of the coordinate axes $X=0$ or $Y=0$.
\item If $m$ is odd and $a_m>0$, then $(0,0)$ is a saddle point.
\item If $m$ is odd and $a_m<0$, then $(0,0)$ is a stable node.
\item If $m$ is even, then $(0,0)$ is a saddle-node. More precisely, there exist two distinct solution curves tangent respectively to the positive and negative $Y$-axis and entering $(0,0)$. These two curves separate $B(0,\delta)$ into a parabolic sector and two hyperbolic sectors. If $a_m>0$, the parabolic sector lies on the left side and the two hyperbolic sectors lie on the right side (see Figure 10(a)); if $a_m<0$, the two hyperbolic sectors lie on the left side and the parabolic sector lies on the right side (see Figure 10(b)).
\end{enumerate}
\end{lemma}

\begin{figure}[htbp]
	\centering
	\captionsetup[subfigure]
	{labelformat=parens}
	\begin{subfigure}{0.43\textwidth}
		\centering
		\includegraphics[width=\linewidth]{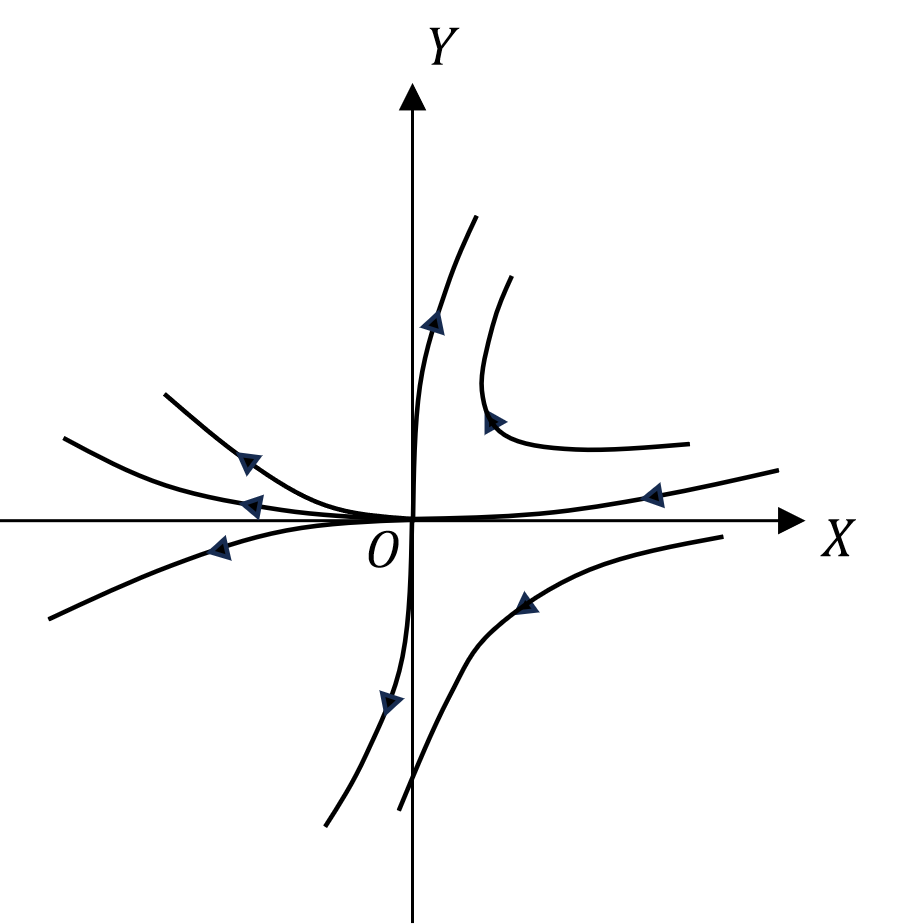}
		\caption{$a_m>0$.}
		\label{3.7:sn:a}
	\end{subfigure}
	\hfill 
	\begin{subfigure}{0.43\textwidth}
		\centering
		\includegraphics[width=\linewidth]{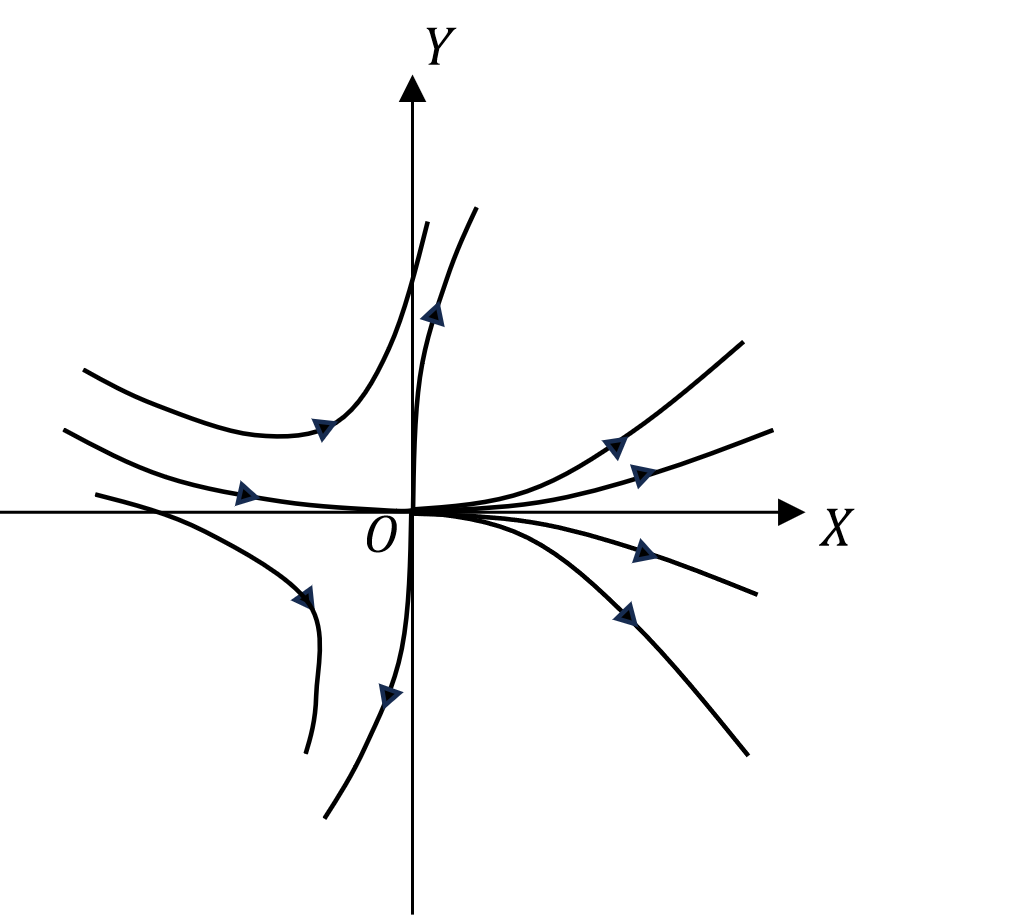}
		\caption{$a_m<0$.}
		\label{3.7:sn:b}
	\end{subfigure}
	\caption{Local phase portrait near a saddle-node equilibrium.}
	\label{3.7:saddle-node}
\end{figure}

Let $W_2=\phi(W_1)$ be defined by $\lambda_3\phi+g_2(W_1,\phi)=0$. A direct expansion gives
\begin{equation}
g_1(W_1,\phi(W_1))
=\dfrac{R\rho_+\gamma(\gamma+1)}{R\gamma\mu+\kappa(\gamma-1)^2}W_1^2+o(W_1^2).
\label{eq:3.36}
\end{equation}
Since the coefficient in \eqref{eq:3.36} is positive, Lemma~\ref{lem:saddle-node-negative} applies with $m=2$ and $a_2>0$. Thus $O$ is a saddle-node. More precisely, in the $(W_1,W_2)$-plane a sufficiently small neighborhood of $O$ is divided by two distinct solution curves tangent to the positive and negative $W_2$-axis and entering $O$; the left side is the parabolic sector and the right side consists of two hyperbolic sectors. Returning to the $(\bar{u}^{\mB},\bar\theta^{\mB})$-plane, the two separating incoming curves are tangent to
\begin{equation}
\tau_2':=\left\{(\bar{u}^{\mB},\bar \theta^{\mB})\;|\;\mu u_+\bar{u}^{\mB}-\kappa(\gamma-1)\bar\theta^{\mB}=0\right\}.
\label{eq:3.37}
\end{equation}
We denote the curve entering $O$ from the fourth quadrant by $\Gamma_3'$ and the curve entering from the second quadrant by $\Gamma_4'$.

We next continue the two boundary curves backward from $O$. The curve $\Gamma_3'$ enters the same region as $\Gamma_1'$ in the subsonic case. By a similar analysis, we see that the backward continuation of $\Gamma_3'$ must meet $L_0$ at a unique point $Z_3$. Every point on the open segment between $Z_3$ and $O$ then generates a boundary layer solution converging to $O$ as $x\to+\infty$.

The curve $\Gamma_4'$ initially enters the upper-left region. The signs of $H_1$ and $H_2$ are the same as in the subsonic analysis of $\Gamma_2'$, so the same continuation cases and comparison estimate apply. Thus situation (B) occurs when $\zeta\leq0$, whereas situation (E) occurs when $\zeta>0$. In the latter case, $\Gamma_4'$ meets $L_0$ at a unique point $Z_4$. In either case, every admissible point on $\Gamma_4'$ generates, in the forward direction, a boundary layer solution approaching the saddle-node $O$.

Let $\Sigma_1'$ be the open region enclosed by $\Gamma_3'$, $\Gamma_4'$ and $L_0$; see Figure~\ref{fig11}. A forward solution starting in $\Sigma_1'$ cannot cross either boundary curve, by uniqueness, and it cannot cross $L_0$ while remaining physical. It therefore stays in the region enclosed by $\Gamma_3'\cup\Gamma_4'\cup L_0$. The nullclines divide this region into finitely many parts with fixed signs of $H_1$ and $H_2$. Table 2 and Table 3 give the directions of the vector field and forward crossing directions on the nullclines of a solution curve. We can see that the crossing directions of solution curves are one-way, so a solution curve can cross the nullclines only finitely many times. It eventually remains in one part, where both components are monotone and bounded, and hence converges to a finite limit $Q$. Passing to the limit in \eqref{eq:3.7} gives $H_1(Q)=H_2(Q)=0$. The only equilibrium accessible from the interior of $\Sigma_1'$ is $O$, and therefore $(\bar{u}^{\mB},\bar\theta^{\mB})(x)\to O$ as $x\to+\infty$. Conversely, every non-trivial solution curve approaching $O$ from this side lies either on one of the two boundary curves or in the region between them. Translating $\Gamma_3'$, $\Gamma_4'$ without endpoints and $\Sigma_1'$ by $(u_+,\theta_+)$ gives $\Gamma_3$, $\Gamma_4$ and $\Sigma_1$. This proves the transonic part of the theorem.

\begin{table}[htbp]
	\centering
	\caption{Directions of the vector field in the transonic invariant region
		\(\Sigma_1'\).}
	\label{tab:transonic-vector-field}
	\begin{tabular}{cccc}
		\toprule
		Position of the point
		& \((\bar u^B)'\)
		& \((\bar\theta^B)'\)
		& Forward direction
		\\
		\midrule
		\(\displaystyle
		\bar\theta^B>h_2(\bar u^B)
		\)
		&
		+
		&
		--
		&
		right and downward
		\\[2mm]
		\(\displaystyle
		h_1(\bar u^B)
		<
		\bar\theta^B
		<
		h_2(\bar u^B)
		\)
		&
		+
		&
		+
		&
		right and upward
		\\[2mm]
		\(\displaystyle
		\bar\theta^B<h_1(\bar u^B)
		\)
		&
		--
		&
		+
		&
		left and upward
		\\
		\bottomrule
	\end{tabular}
\end{table}
\begin{table}[htbp]
	\centering
	\caption{Forward crossing directions on the nullclines in the
		transonic case \(M_+=1\).}
	\label{tab:transonic-crossing}
	\begin{tabular}{ccc}
		\toprule
		Nullcline
		& Forward crossing direction
		\\
		\midrule
		\(\displaystyle
		\bar\theta^B=h_1(\bar u^B)
		\)
		&
		\(\displaystyle
		\bar\theta^B<h_1(\bar u^B)
		\ \longrightarrow\
		\bar\theta^B>h_1(\bar u^B)
		\)
		\\[3mm]
		\(\displaystyle
		\bar\theta^B=h_2(\bar u^B)
		\)
		&
		\(\displaystyle
		\bar\theta^B<h_2(\bar u^B)
		\ \longrightarrow\
		\bar\theta^B>h_2(\bar u^B)
		\)
		\\
		\bottomrule
	\end{tabular}
\end{table}
\begin{figure}[htbp]
	\centering
	\begin{subfigure}{0.7\textwidth}
		\centering
		\includegraphics[width=\linewidth]{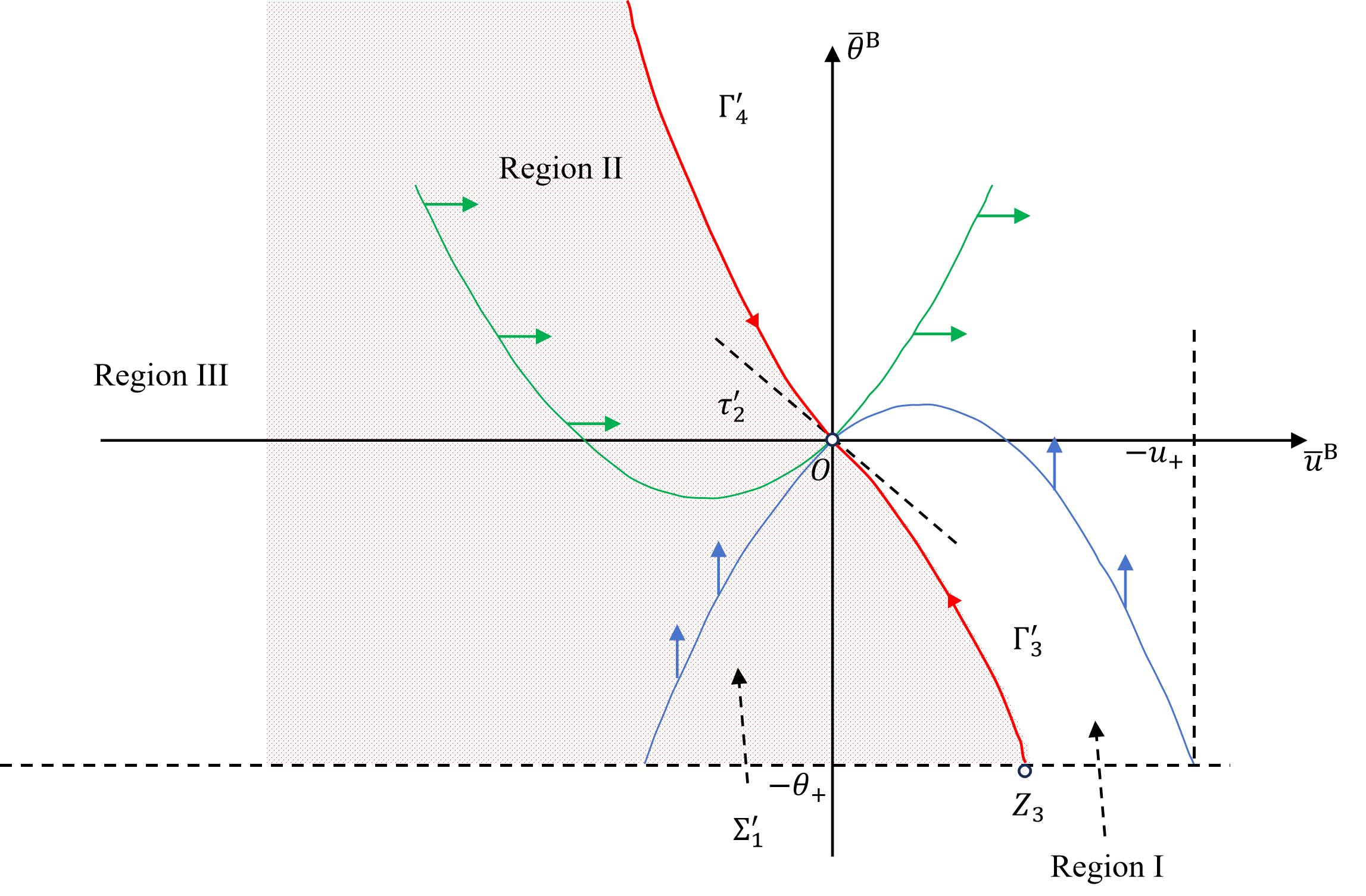}
		\caption{$\zeta\leq0.$}
		\label{fig11(1)}
	\end{subfigure}
	\hfill 
	\begin{subfigure}{0.7\textwidth}
		\centering
		\includegraphics[width=\linewidth]{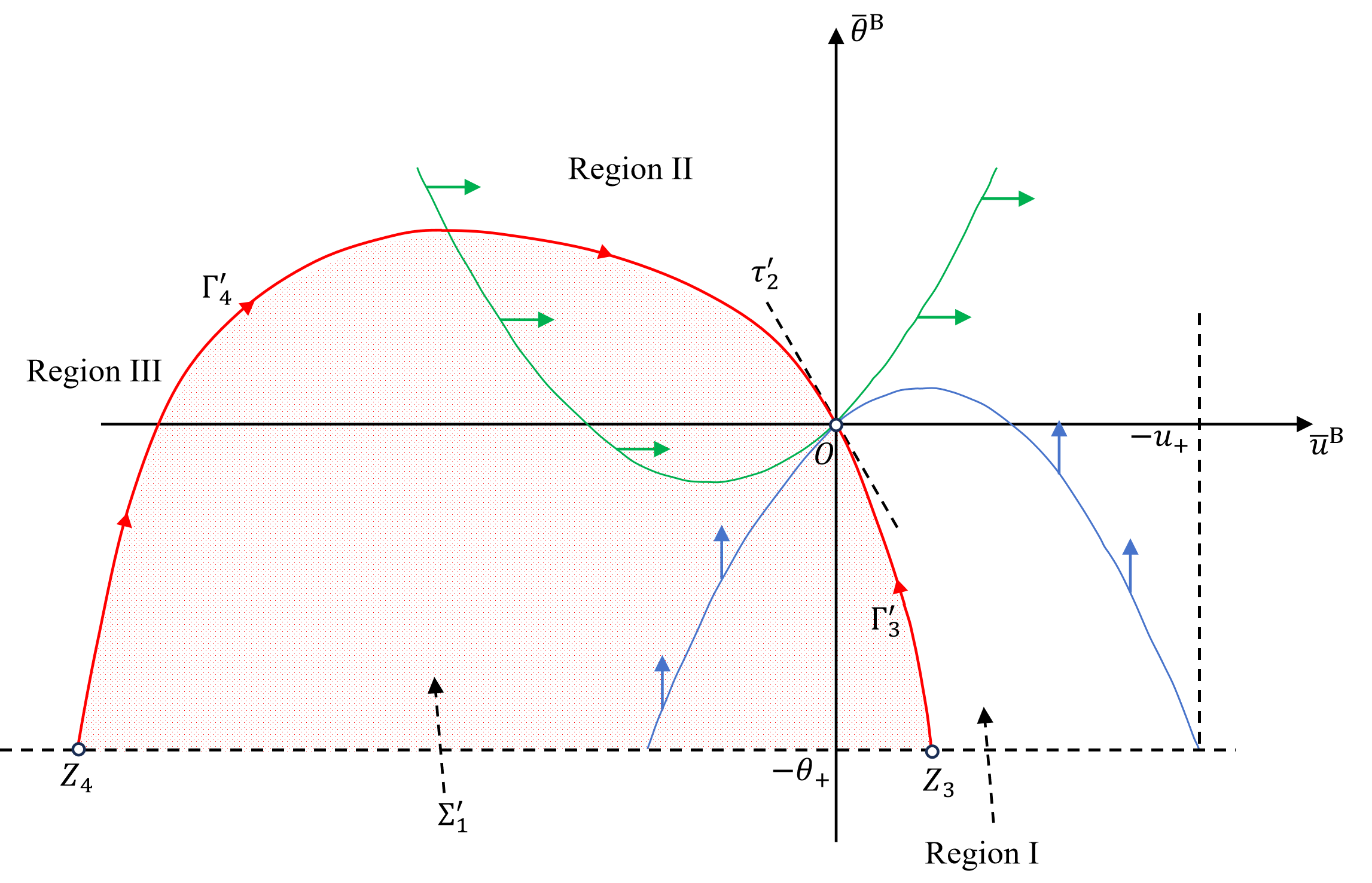}
		\caption{$\zeta>0.$}
		\label{fig11(2)}
	\end{subfigure}
	\caption{The shifted phase portraits in the transonic outflow case.}
	\label{fig11}
\end{figure}

\subsection{Supersonic case: $\Mplus>1$}

For $\Mplus>1$, the origin is a stable node. Thus a small neighborhood of $O$ is contained in the basin of initial states whose solution curves satisfy $(\bar{u}^{\mB},\bar\theta^{\mB})(+\infty)=(0,0)$. To determine the large-amplitude boundary of this basin, we use the second equilibrium of \eqref{eq:3.7}, namely
\begin{equation}
\begin{aligned}
P&=(\alpha_1u_+,\alpha_2\theta_+),\qquad
\alpha_1:=\dfrac{2(1-\Mplus^2)}{\Mplus^2(\gamma+1)}\in(-1,0),\\
\alpha_2&:=\dfrac{2(\Mplus^2-1)(\Mplus^2\gamma+1)(\gamma-1)}{\Mplus^2(\gamma+1)^2}>0.
\end{aligned}
\label{eq:3.38}
\end{equation}
Linearizing the vector field at $P$ gives the matrix
\begin{equation}
\widetilde A=
\begin{pmatrix}
\displaystyle
\dfrac{(\Mplus^2\gamma^2-3\Mplus^2\gamma+3\gamma-1)\rho_+u_+}{(\Mplus^2(\gamma-1)+2)\mu\gamma}
&
\displaystyle
\dfrac{R\rho_+\Mplus^2(\gamma+1)}{(\Mplus^2(\gamma-1)+2)\mu}
\\[3mm]
\displaystyle
\dfrac{\rho_+u_+^2(2\Mplus^2\gamma-\gamma+1)}{\Mplus^2\gamma(\gamma+1)\kappa}
&
\displaystyle
\dfrac{R\rho_+u_+}{\kappa(\gamma-1)}
\end{pmatrix},
\label{eq:3.39}
\end{equation}
A direct computation yields
\begin{equation}
\det \widetilde A
=-\dfrac{R\rho_+^2u_+^2(\gamma+1)(\Mplus^2-1)}{\kappa\mu(\gamma-1)(\Mplus^2(\gamma-1)+2)}<0,
\qquad \Mplus>1.
\label{eq:3.40}
\end{equation}
Hence $P$ is a saddle. Let $\lambda_4<0$ be the negative eigenvalue of $\widetilde A$. The two stable separatrices entering $P$ are denoted by $\Gamma_5'$ and $\Gamma_6'$, where $\Gamma_5'$ enters from the right-lower side and $\Gamma_6'$ from the left-upper side. They are tangent at $P$ to the straight line $\tau_3'$, which satisfies
\begin{equation}
	\left(\dfrac{(\Mplus^2\gamma^2-3\Mplus^2\gamma+3\gamma-1)\rho_+u_+}{(\Mplus^2(\gamma-1)+2)\mu\gamma}-\lambda_4\right)(\bar{u}^{\mB}-\alpha_1u_+)
+\dfrac{R\rho_+\Mplus^2(\gamma+1)}{(\Mplus^2(\gamma-1)+2)\mu}(\bar\theta^{\mB}-\alpha_2\theta_+)=0.
\label{eq:3.41}
\end{equation}

Define
\begin{equation}
l_2:=\{(\bar{u}^{\mB},\bar\theta^{\mB}):\bar\theta^{\mB}=h_1(\bar{u}^{\mB}),\ \alpha_1u_+\le\bar{u}^{\mB}\le -u_+\}.
\label{eq:3.42}
\end{equation}
Let Region IV be the region enclosed by $\bar{u}^{\mB}=\alpha_1u_+$, $L_0$ and $l_2$. On $l_2$ one has $H_1=0$, and the sign of $H_2$ shows that the vector field points into the region when the curve is traced backward from $P$. Along the relevant part of the vertical boundary $\bar{u}^{\mB}=\alpha_1u_+$, one has $\bar\theta^{\mB}<h_1(\alpha_1u_+)$ and hence $H_1<0$. The backward vector field therefore points into Region IV, so the separatrix cannot leave through this boundary. Since Region IV contains no equilibrium point except the endpoint $P$, the curve $\Gamma_5'$ must leave the region through the physical boundary $L_0$; denote the intersection point by $Z_5$. This curve is used as one component of the boundary of the admissible basin. Points on $\Gamma_5'$ themselves are not counted as admissible outflow boundary data in the supersonic case, because their positive $x$-extensions approach the saddle $P$ rather than the far-field equilibrium $O$.

We define Region V and Region VI as
\begin{equation*}
	\mathrm{V}:=\{(\bar{u}^{\mB},\bar\theta^{\mB})\;|\;\bar\theta^{\mB}> \max\{h_1(\bar{u}^{\mB}),h_2(\bar{u}^{\mB})\},\bar{u}^{\mB}\leq\alpha_1u_+\},
\end{equation*}
\begin{equation*}
	\mathrm{VI}:=\{(\bar{u}^{\mB},\bar\theta^{\mB})\;|\;h_1(\bar{u}^{\mB})\le \bar\theta^{\mB}\le h_2(\bar{u}^{\mB}),\ \bar{u}^{\mB}\le 0,\ \bar\theta^{\mB}>-\theta_+\}.
\end{equation*}
The curve $\Gamma_6'$ is the analogue of the curve $\Gamma_2'$. Since the signs of $H_1$ and $H_2$ are unchanged, the same continuation cases and comparison estimate apply. Thus situation (B) occurs when $\zeta\leq0$, whereas situation (E) occurs when $\zeta>0$. In the latter case, $\Gamma_6'$ meets $L_0$ at a unique point $Z_6$. This separatrix forms the other boundary of the basin. As for $\Gamma_5'$, points on $\Gamma_6'$ approach $P$ in the positive $x$-direction and therefore do not generate boundary layer solutions with far-field state $O$.

Let $\Sigma_2'$ be the open region enclosed by $\Gamma_5'$, $\Gamma_6'$ and $L_0$; see Figure~\ref{fig12}. A forward solution starting in $\Sigma_2'$ cannot cross the two stable separatrices, by uniqueness, and it cannot cross $L_0$ while remaining physical. Since the singular line $\mathcal S$ lies outside the region under consideration, the solution remains in the region enclosed by $\Gamma_5'\cup\Gamma_6'\cup L_0$. The nullclines divide this region into finitely many parts with fixed signs of $H_1$ and $H_2$. Table 4 and Table 5 give the directions of the vector field and forward crossing directions on the nullclines of a solution curve. Similarly, the crossing directions of solution curves are one-way, so the solution crosses the nullclines only finitely many times and eventually remains in one part. Both components are then monotone and bounded, and the solution converges to a finite limit $Q$. Passing to the limit in \eqref{eq:3.7} yields $H_1(Q)=H_2(Q)=0$. The saddle $P$ can be approached only along its stable separatrices, which form the boundary of $\Sigma_2'$. Hence an interior solution cannot converge to $P$, and its limit must be the stable node $O$.

Translating $\Gamma_5'$, $\Gamma_6'$ without endpoints and $\Sigma_2'$ by $(u_+,\theta_+)$ gives $\Gamma_5$, $\Gamma_6$ and $\Sigma_2$ in the original variables. Under our convention, the far-field point $(u_+,\theta_+)$, which gives the trivial constant solution, is excluded. Consequently, in the supersonic case a boundary layer solution exists if and only if
\[
(u_-,\theta_-)\in\Sigma_2\setminus\{(u_+,\theta_+)\}.
\]

Since the boundary layer solution converges to the far-field state,
there exists $X_0>0$ such that the boundary layer solution remains in a sufficiently
small neighborhood of the far-field equilibrium for all $x\ge X_0$.
The local invariant-manifold analysis can therefore be applied to the
tail of the boundary layer solution. The resulting exponential decay in the
nondegenerate cases and algebraic decay in the transonic case follow
as in \cite{KawashimaNakamuraNishibataZhu2010}. This completes the proof of Theorem~\ref{thm:main}.
\newpage

\begin{table}[htbp]
	\centering
	\renewcommand{\arraystretch}{1.4}
	\caption{Directions of the vector field in the supersonic
		invariant region \(\Sigma_2'\).}
	\label{tab:supersonic-vector-field}
	\begin{tabular}{cccc}
		\toprule
		Position of the point
		& \((\bar u^B)'\)
		& \((\bar\theta^B)'\)
		& Direction
		\\
		\midrule
		\(\displaystyle
		\bar\theta^B>
		\max\left\{
		h_1(\bar u^B),h_2(\bar u^B)
		\right\}
		\)
		&
		\(\displaystyle >0\)
		&
		\(\displaystyle <0\)
		&
		right and downward
		\\[2mm]
		\(\displaystyle
		h_1(\bar u^B)
		<
		\bar\theta^B
		<
		h_2(\bar u^B)
		\)
		&
		\(\displaystyle >0\)
		&
		\(\displaystyle >0\)
		&
		right and upward
		\\[2mm]
		\(\displaystyle
		h_2(\bar u^B)
		<
		\bar\theta^B
		<
		h_1(\bar u^B)
		\)
		&
		\(\displaystyle <0\)
		&
		\(\displaystyle <0\)
		&
		left and downward
		\\[2mm]
		\(\displaystyle
		\bar\theta^B<
		\min\left\{
		h_1(\bar u^B),h_2(\bar u^B)
		\right\}
		\)
		&
		\(\displaystyle <0\)
		&
		\(\displaystyle >0\)
		&
		left and upward
		\\
		\bottomrule
	\end{tabular}
\end{table}
\begin{table}[htbp]
	\centering
	\renewcommand{\arraystretch}{1.5}
	\caption{Crossing directions on the nullclines in the supersonic
		case \(M_+>1\).}
	\label{tab:supersonic-crossing}
	\begin{tabular}{ccc}
		\toprule
		Nullcline
		& Range of \(\bar u^B\)
		& Direction
		\\
		\midrule
		\(\displaystyle
		\bar\theta^B=h_1(\bar u^B)
		\)
		&
		\(\displaystyle
		\bar u^B<0
		\)
		&
		from below to above
		\\[2mm]
		\(\displaystyle
		\bar\theta^B=h_1(\bar u^B)
		\)
		&
		\(\displaystyle
		0<\bar u^B<\alpha_1u_+
		\)
		&
		from above to below
		\\[2mm]
		\(\displaystyle
		\bar\theta^B=h_2(\bar u^B)
		\)
		&
		\(\displaystyle
		\frac{u_+}{M_+^2\gamma}
		<
		\bar u^B
		<
		0
		\)
		&
		from above to below
		\\[2mm]
		\(\displaystyle
		\bar\theta^B=h_2(\bar u^B)
		\)
		&
		\(\displaystyle
		0<\bar u^B<\alpha_1u_+
		\)
		&
		from below to above
		\\
		\bottomrule
	\end{tabular}
\end{table}
\begin{figure}[htbp]
	\centering
	\begin{subfigure}{0.7\textwidth}
		\centering
		\includegraphics[width=\linewidth]{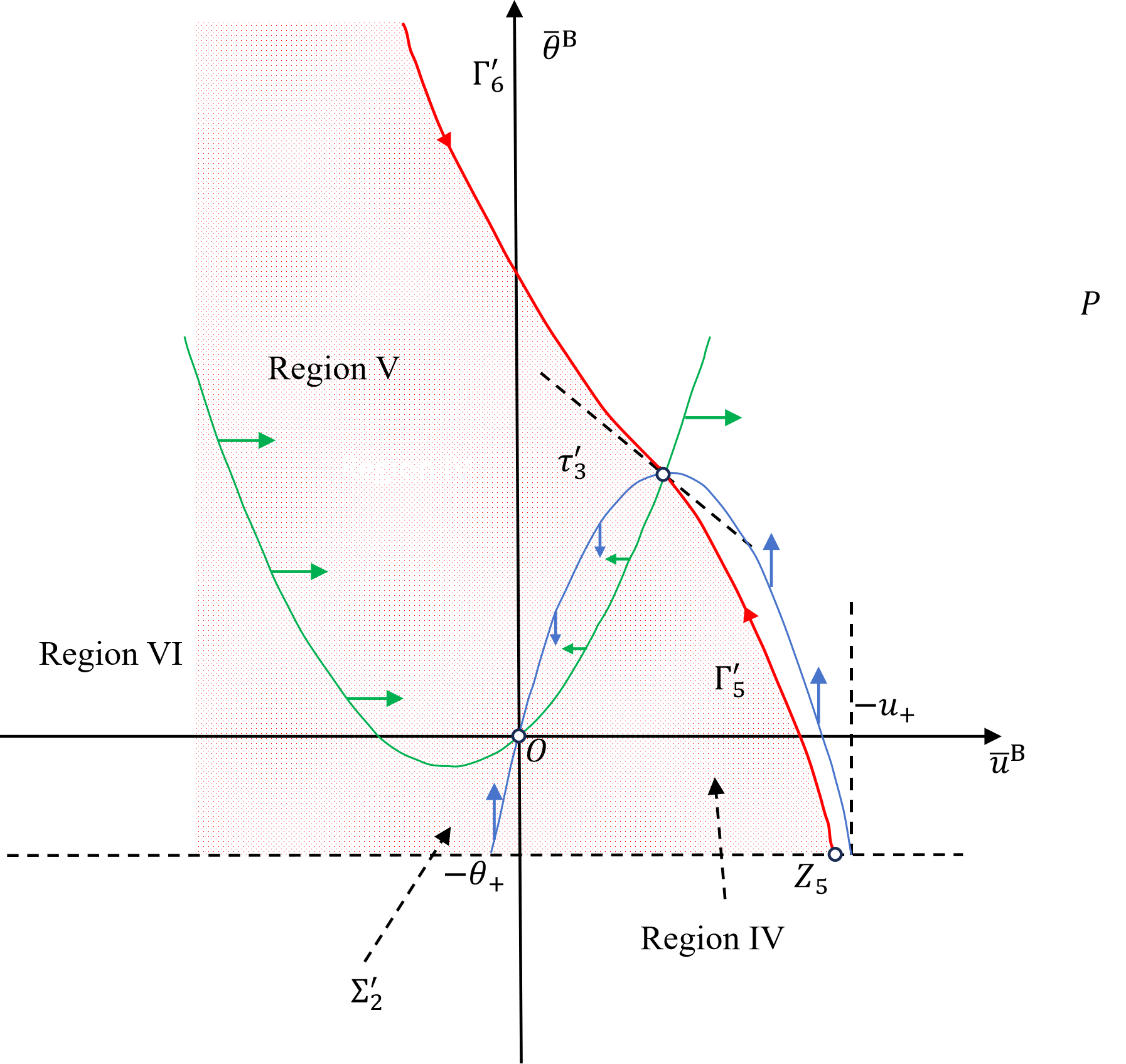}
		\caption{$\zeta\leq0.$}
		\label{fig12(1)}
	\end{subfigure}
	\hfill 
	\begin{subfigure}{0.7\textwidth}
		\centering
		\includegraphics[width=\linewidth]{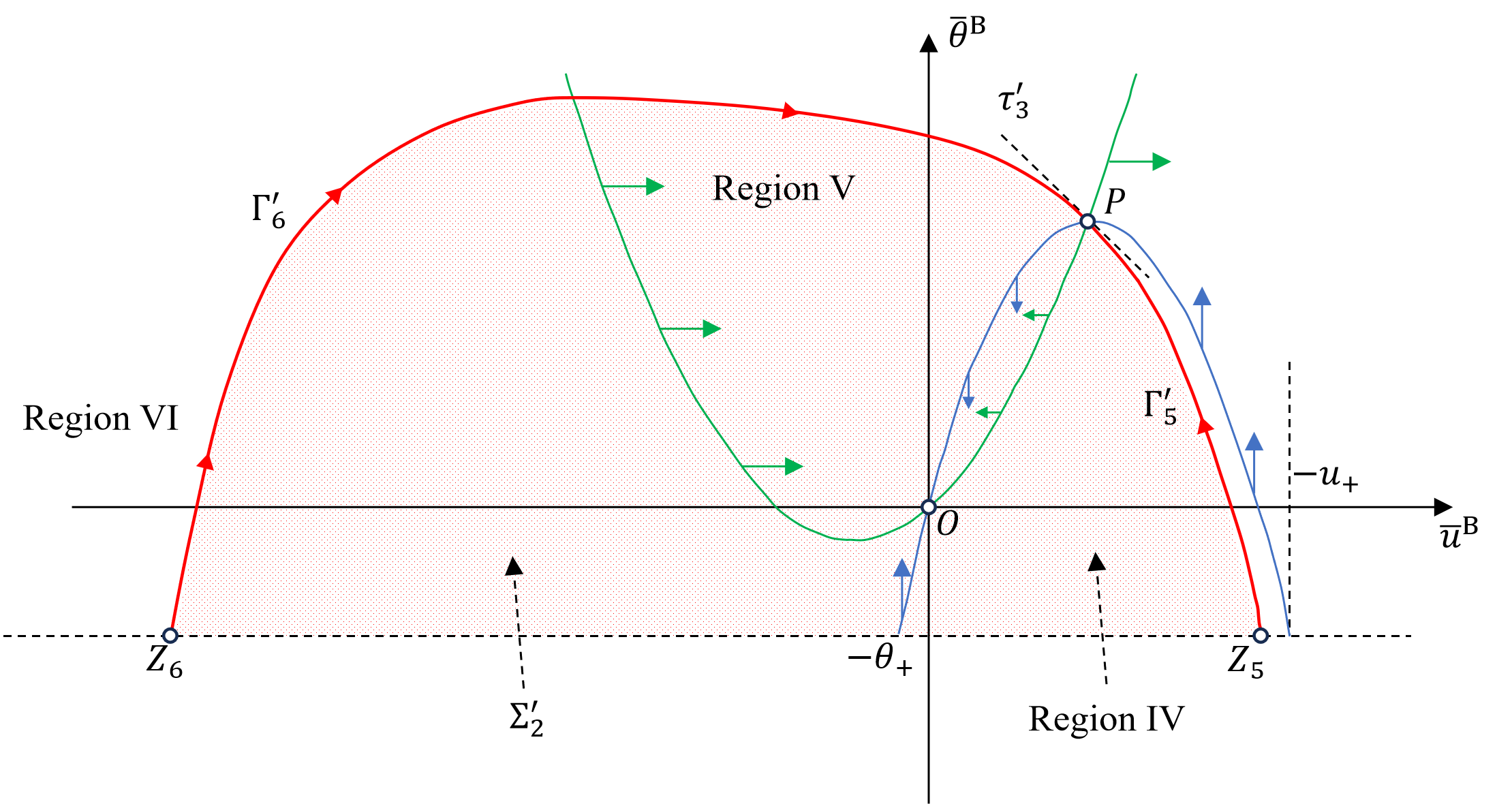}
		\caption{$\zeta>0.$}
		\label{fig12(2)}
	\end{subfigure}
	\caption{The shifted phase portraits in the supersonic outflow case.}
	\label{fig12}
\end{figure}


Finally, we make the following two remarks.
\begin{remark}

Theorem~\ref{thm:main} completely classifies the existence and non-existence of large-amplitude outflow boundary layer solutions to \eqref{eq:2.1}. Nevertheless, whether these large-amplitude boundary layer solutions are time-asymptotically stable for the initial–boundary value problem \eqref{eq:1.1}–\eqref{eq:1.4} remains largely open. The existing literature \cite{KawashimaNishibataZhu2003,NakamuraNishibataYuge2007,KageiKawashima2006,KawashimaNakamuraNishibataZhu2010,Qin2011,WanWangZou2016,ChenHongShi2019,ChenHongShi2021} is based on the relative energy methods which crucially rely on small boundary layer amplitudes. As such, these results cannot be applied to the large-amplitude solutions established in Theorem~\ref{thm:main}. Specifically, the large-amplitude boundary layer solutions established in Theorem~\ref{thm:main} may exhibit non-monotonicity. This feature severely limits standard energy methods, which generally require profile monotonicity or small gradients. 
\end{remark}

\begin{remark}
    We have established a global phase-plane classification of large-amplitude boundary layer solutions for the outflow problem governed by the one-dimensional full compressible Navier-Stokes equations \eqref{eq:1.1}. Compared to the inflow counterpart in \cite{WangYangYu2025}, the outflow problem exhibits distinct characteristics, including a potential loss of monotonicity, different backward continuations of the relevant trajectories, and the emergence of an additional saddle equilibrium in the supersonic case. By utilizing the comparison estimate for the quantity $\alpha={\bar\theta^{\mB}}/{(\bar u^{\mB})^2}$, we categorize the continuation patterns based on the sign of the parameter $\zeta$ defined in \eqref{eq:2.2}. Finally, invariant-region arguments are applied to determine the set of admissible boundary data.
\end{remark}

\end{document}